\documentclass{article}
\pdfoutput=1

\usepackage[hidelinks]{hyperref}

\usepackage[figuresright]{rotating}
\usepackage{floatpag}
\rotfloatpagestyle{empty}

\usepackage{amsmath}
\usepackage{amsthm}
\usepackage{amsfonts}
\usepackage{amssymb}
\usepackage{mathrsfs}

\usepackage{colonequals}
\usepackage{tikz-cd}
\usepackage{xypic}
\usepackage{enumitem}

\usepackage{microtype}

\theoremstyle{plain}
\newtheorem{theorem}{Theorem}[section]
\newtheorem{lemma}[theorem]{Lemma}
\newtheorem{proposition}[theorem]{Proposition}
\newtheorem{corollary}[theorem]{Corollary}

\newtheorem*{theorem*}{Theorem}
\newtheorem*{lemma*}{Lemma}
\newtheorem*{proposition*}{Proposition}
\newtheorem*{corollary*}{Corollary}
\newtheorem*{conjecture*}{Conjecture}

\theoremstyle{definition}
\newtheorem{definition}[theorem]{Definition}
\newtheorem{example}[theorem]{Example}
\newtheorem{remark}[theorem]{Remark}

\newtheorem*{definition*}{Definition}
\newtheorem*{example*}{Example}
\newtheorem*{remark*}{Remark}
\newtheorem*{notation*}{Notation}

\usepackage{cleveref}

\usepackage{adjustbox}
\usepackage{booktabs}
\usepackage{tikz-cd}
\usepackage{colonequals}

\newcommand{\splink}[1]{\href{https://stacks.math.columbia.edu/tag/#1}{\textsc{#1}}}
\newcommand{\spref}[1]{Tag \href{https://stacks.math.columbia.edu/tag/#1}{\textsc{#1}}}

\newcommand{\sM}{\mathcal{M}}

\newcommand{\al}{\alpha}
\newcommand{\sF}{\mathcal{F}}
\newcommand{\eps}{\epsilon}
\newcommand{\sU}{\mathcal{U}}
\newcommand{\p}{\mathbb{P}}
\newcommand{\sP}{\mathcal{P}}
\newcommand{\subs}{\subseteq}
\newcommand{\om}{\omega}

\newcommand\univ{\ensuremath{\mathrm{univ}}}
\newcommand\universalbundle{\ensuremath{\mathcal{E}_{\univ}}}

\DeclareMathOperator\Hom{Hom}
\DeclareMathOperator\Ext{Ext}
\DeclareMathOperator\rk{rk}
\DeclareMathOperator\Spec{Spec}

\DeclareMathOperator\GL{GL}
\DeclareMathOperator\gr{gr}
\DeclareMathOperator\Sym{Sym}
\DeclareMathOperator\id{id}

\DeclareMathOperator\Aut{Aut}
\DeclareMathOperator\HH{H}
\DeclareMathOperator\hh{h}

\DeclareMathOperator\stPic{\mathcal{P}ic}

\DeclareMathOperator\tr{tr}
\DeclareMathOperator\Pic{Pic}

\DeclareMathOperator\Kzero{K_0}

\newcommand\modulistack[4][]{\mathcal{M}_{#2}^{#1}(#3,#4)}
\newcommand\modulispace[4][]{\mathrm{M}_{#2}^{#1}(#3,#4)}

\newcommand\semistable{\mathrm{ss}}
\newcommand\stable{\mathrm{s}}

\begin{document}

\title{Projectivity of the moduli space of vector bundles on a curve}
\author{Jarod Alper \and Pieter Belmans \and Daniel Bragg \and Jason Liang \and Tuomas Tajakka}

\maketitle

\begin{abstract}
  We discuss the projectivity of the moduli space of semistable vector bundles on a curve of genus $g\geq 2$. This is a classical result from the 1960s, obtained using geometric invariant theory. We outline a modern approach that combines the recent machinery of good moduli spaces with determinantal line bundle techniques. The crucial step producing an ample line bundle follows an argument by Faltings with improvements by Esteves--Popa. We hope to promote this approach as a blueprint for other projectivity arguments.
\end{abstract}

\begin{center}
  \emph{Dedicated to the memory of \\ Conjeevaram~Srirangachari~Seshadri}
\end{center}

In \cite{MR1064874}, Koll\'ar took a modern approach to the construction of the moduli space~$\overline{\mathrm{M}}_g$ of stable curves of genus~$g\geq 2$ as a projective variety, avoiding methods from geometric invariant theory. The approach can be summarized in three steps.
\begin{enumerate}
  \item Prove that the stack of stable curves is a proper Deligne--Mumford stack $\overline{\mathcal{M}}_{g}$ \cite{MR0262240}.

  \item Use the Keel--Mori theorem to show that $\overline{\mathcal{M}}_{g}$ has a coarse moduli space $\overline{\mathcal{M}}_{g} \to \overline{\mathrm{M}}_{g}$, where $\overline{\mathrm{M}}_{g}$ is a proper algebraic space \cite{MR1432041}.

  \item\label{enumerate:Mg-ample} Prove that some line bundle on $\overline{\mathcal{M}}_{g}$ descends to an \emph{ample} line bundle on $\overline{\mathrm{M}}_{g}$ \cite{MR1064874,MR1241126}.
\end{enumerate}
In \cite{maulik} the projectivity of $\overline{\mathrm{M}}_{g}$ 
was established by giving a stack-theoretic treatment of Koll\'ar's paper. 

We take a similar approach to constructing the moduli space $\modulispace[\semistable]{X}{r}{d}$ of semistable vector bundles of rank $r$ and degree $d$ on a smooth, proper curve $X$ as a projective variety. There are again three steps in the proof.
\begin{enumerate}
  \item\label{enumerate:MXrd-properties} Prove that $\modulistack[\semistable]{X}{r}{d}$ is a universally closed algebraic stack of finite type \cite{MR0364255}.

  \item\label{enumerate:MXrd-gms} Use an analogue of the Keel--Mori theorem to show that there exists a \emph{good moduli space} $\modulistack[\semistable]{X}{r}{d} \to \modulispace[\semistable]{X}{r}{d}$ \cite{1812.01128v3}, where~$\modulispace[\semistable]{X}{r}{d}$ is a proper algebraic space.

    Unlike the case of~$\overline{\mathcal{M}}_g$, the stack $\modulistack[\semistable]{X}{r}{d}$ is never Deligne--Mumford and the morphism to~$\modulispace[\semistable]{X}{r}{d}$ will identify nonisomorphic vector bundles, introducing the notion of S-equivalence.

  \item\label{enumerate:MXrd-ample} Prove that a suitable determinantal line bundle on $\modulistack[\semistable]{X}{r}{d}$ descends to an ample line bundle on $\modulispace[\semistable]{X}{r}{d}$.
\end{enumerate}

In Sections \ref{section:moduli-all}--\ref{section:good-moduli-space} we define and study the moduli stack $\modulistack[\semistable]{X}{r}{d}$ and establish steps \ref{enumerate:MXrd-properties} and \ref{enumerate:MXrd-gms}. The main points are summarized in Theorems \ref{theorem:properties} and \ref{theorem:good-moduli-space}.


In Section \ref{section:projectivity} we discuss step \ref{enumerate:MXrd-ample} and give a detailed exposition of the construction of an ample line bundle on $\modulispace[\semistable]{X}{r}{d}$ using techniques from Esteves--Popa \cite{MR2068965}. This method is more recent than the first GIT-free construction due to Faltings \cite{MR1211997}, which has received expository accounts by Seshadri (and Nori) in \cite{MR1247504} and for a special case by Hein in \cite{MR3013028}. All these constructions build upon the notion of determinantal line bundles, which we introduce in Section~\ref{section:determinantal}.

Because the existence results for good moduli spaces from \cite{1812.01128v3} are currently restricted to characteristic~0, we make that restriction as well. 

\section{The moduli stack of all vector bundles}
\label{section:moduli-all}
Let $k$ be an algebraically closed field of characteristic $0$, and let $X$ be a smooth, projective, connected curve over $k$.

We denote by $\underline{{\rm Coh}}_X$ the stack of all coherent sheaf on $X$. For integers $r>0$ and $d$, we let $\modulistack{X}{r}{d}$ denote the stack parameterizing vector bundles on $X$ of rank $r$ and degree $d$. An object of $\modulistack{X}{r}{d}$ over a scheme $S$ is a vector bundle $\mathcal{E}$ on $X \times S$ such that for every geometric point $\Spec K \to S$, the restriction $\mathcal{E}|_{X \times \Spec K}$ has rank $r$ and degree $d$.

\begin{proposition}
  \label{proposition:affine-diagonal}
  The stack $\modulistack{X}{r}{d}$ is an algebraic stack, locally of finite type over $k$ with affine diagonal.
\end{proposition}

\begin{proof}
  The stack $\underline{{\rm Coh}}_X$ is algebraic by \cite[\spref{09DS}]{stacks-project}. Given a map $T \to \underline{{\rm Coh}}_X$ corresponding to a sheaf $\mathcal{E}$ on $T \times X$ flat over $T$, the locus $Z \subs T \times X$ where $\mathcal{E}$ is not locally free is closed, so $U = T \setminus \mathrm{pr}_1(Z) \subs T$ is open by properness, showing that the substack $\mathcal{M} \subset \underline{{\rm Coh}}_X$ parameterizing vector bundles on $X$ is an open substack. The Hilbert polynomial~$\mathrm{P}_E$ of a coherent sheaf $E$ on $X$ is determined by its rank and degree, and by the Riemann--Roch theorem it is given by
  \begin{equation}
    \mathrm{P}_E(n) = \rk(E) n + \deg(E) + \rk(E) (1 - g). 
  \end{equation}
  It is locally constant in flat families, so by the discussion in \cite[\spref{0DNF}]{stacks-project} there is an open and closed substack $\modulistack{X}{r}{d} \subset \mathcal{M}$ parameterizing vector bundles of rank $r$ and degree $d$.

  Since $\underline{{\rm Coh}}_X$ is locally of finite type over $k$ by \cite[\spref{0DLZ}]{stacks-project}, and has affine diagonal by \cite[\spref{0DLY}]{stacks-project}, so does $\modulistack{X}{r}{d}$.
\end{proof}

There is a universal vector bundle $\universalbundle$ on $X \times \modulistack{X}{r}{d}$. Specializing to the rank-1 case, we obtain the Picard stack $\stPic_X^d$ parameterizing line bundles of degree $d$ on $X$. In this case the universal bundle is called the Poincar\'e bundle and denoted $\mathcal{P}.$ The line bundle $\det(\universalbundle)$ on $X \times \modulistack{X}{r}{d}$ induces a determinant morphism
\[
  \det\colon \modulistack{X}{r}{d} \to \stPic_X^d,
\]
with the property that $(\id_X \times \det)^* \mathcal{P} \cong \det(\universalbundle)$.

Let $L$ be a line bundle on $X$ of degree $d$ corresponding to a closed point of~$\stPic_X^d$, and let $[L]\colon\Spec k\to \stPic_X^d$ be the corresponding morphism. The algebraic stack parameterizing vector bundles with fixed determinant $L$ is defined as the fiber product
\begin{equation}\label{eq:cartesian square over det morphism}
  \begin{tikzcd}
    \modulistack{X}{r}{L} \arrow[r] \arrow[d] \arrow[rd, phantom, "\square"] &  \modulistack{X}{r}{d} \arrow[d, "\det"]  \\
    \Spec k \arrow[r, "{[L]}"] & \stPic_X^d.
  \end{tikzcd}
\end{equation}
Explicitly, an object of $\modulistack{X}{r}{L}$ over a $k$-scheme $S$ is a pair $(\mathcal{E},\varphi)$, where $\mathcal{E}$ is a vector bundle on $X\times S$ with rank $r$ and degree $d$ on the fibers, and $\varphi\colon\det\mathcal{E}\xrightarrow{\sim}L|_{X\times S}$ is an isomorphism.

An alternative approach to fixing the determinant uses the residual gerbe $\mathrm{B}\mathbb{G}_{\mathrm{m}}\hookrightarrow\stPic_X^d$ instead. We compare these in Example~\ref{example:to-dm-or-not-to-dm}.


\begin{proposition}
  The stack $\modulistack{X}{r}{L}$ is an algebraic stack, locally of finite type over $k$ with affine diagonal.
\end{proposition}

\begin{proof}
  This follows by applying \cite[\spref{04TF}]{stacks-project} for the algebraicity and \cite[\spref{06FU}]{stacks-project} for being locally of finite type. 
  To see that $\modulistack{X}{r}{L}$ has affine diagonal, we use the fact that~$\det\colon\modulistack{X}{r}{d}\to\stPic_X^d$ has affine diagonal (because both stacks do) and that having affine diagonal is stable under base change.
\end{proof}

Let $\mathscr{X}$ be an algebraic stack. Let $x_0\in\mathscr{X}(k)$ be a $k$-point. By \cite[\spref{07WZ}]{stacks-project}, the tangent space $\mathrm{T}_{\mathscr{X},x_0}$ to $\mathscr{X}$ at $x_0$ is the set of isomorphism classes of objects in the groupoid $\mathscr{X}_{x_0}(k[\varepsilon]/(\varepsilon^2))$, where~$\mathscr{X}_{x_0}$ denotes the slice category. Using deformation theory one can prove the following result about the moduli stacks.
\begin{proposition} \ \vspace{-.3cm}
  \label{proposition:smooth}
  \begin{enumerate}
    \item The algebraic stacks $\modulistack{X}{r}{d}$ and $\modulistack{X}{r}{L}$ are smooth over $k$.
    \item If $E$ is a vector bundle of rank $r$ and $\varphi\colon L \cong\det E$ is an isomorphism, then there are identifications of the Zariski tangent space
      \begin{align}
        \mathrm{T}_{\modulistack{X}{r}{d},[E]} & \cong \Ext_X^1(E,E), \\
        \mathrm{T}_{\modulistack{X}{r}{L},[(E,\varphi)]} & \cong \ker\big( \Ext_X^1(E,E) \xrightarrow{\tr} \HH^1(X, \mathcal{O}_X) \big). \label{equation:tangent-space-L-fixed}
      \end{align}
  \end{enumerate}
\end{proposition}

\begin{remark}
  \label{remark:dimension}
  Let $E$ be a \textit{simple} vector bundle of rank $r$ and degree~$d$, that is, $\dim \Hom_X(E,E) = 1$. We will see in Lemma \ref{lemma:stable-endomorphisms-field} that stable bundles are simple. By the Riemann--Roch theorem, we have
  \begin{align*}
      \dim \Ext_X^1(E,E) & = \dim \Hom_X(E,E) - \chi(X, E^\vee \otimes E) \\
      & = 1 - \deg(E^ \vee \otimes E) - \rk(E^\vee \otimes E)(1 - g) \\
      & = r^2(g - 1) + 1.
  \end{align*}
  Next, note that since we are working over a field of characteristic $0$, the trace map $E^\vee \otimes E \to \mathcal{O}_X$ is split by the section $\frac{1}{d} \id_E\colon \mathcal{O}_X \to E^\vee \otimes E \cong \Hom_X(E, E)$. Thus, the surjection $\Ext_X^1(E, E) \to \HH^1(X, \mathcal{O}_X)$ in \eqref{equation:tangent-space-L-fixed} is split, and so its kernel has dimension
  \begin{align*}
    \dim \Ext_X^1(E, E) - \dim \HH^1(X, \mathcal{O}_X) &= r^2(g - 1) + 1 - g \\
    &= (r^2 - 1)(g - 1).
  \end{align*}
\end{remark}

\section{Semistability}
\label{section:semistability}


In order to obtain a more well-behaved geometric object, we will restrict ourselves to an open substack~$\modulistack[\semistable]{X}{r}{d}$ of the stack $\modulistack{X}{r}{d}$. This substack parameterizes the \emph{semistable} vector bundles. The Jordan--H\"older and Harder--Narasimhan filtrations explain how it is sufficient to consider these vector bundles in order to describe all vector bundles on~$X$:
\begin{itemize}
  \item the Harder--Narasimhan filtration allows one to uniquely break up any vector bundles into semistable constituents;
  \item the Jordan--H\"older filtration allows one to break up a semistable vector bundle into \textit{stable} bundles. 
\end{itemize}
The essential properties of semistability, which will be formalized in Theorem~\ref{theorem:properties}, are
\begin{itemize}
  \item openness: $\modulistack[\semistable]{X}{r}{d}$ is an open substack of $\modulistack{X}{r}{d}$;
  \item boundedness: $\modulistack[\semistable]{X}{r}{d}$ is a quasicompact;
  \item irreducibility of the semistable locus.
\end{itemize}
For more information, in particular on the boundedness of semistable sheaves in higher dimensions and in arbitrary characteristic, see \cite{perry}.

\subsection{Definition and basic properties}

\begin{definition}
  Given a nonzero vector bundle $E$ on $X$, we define the \emph{slope} of $E$ to be the rational number
  \[
    \mu(E)\colonequals \frac{\deg(E)}{\rk(E)}.
  \]
\end{definition}

Recall that on a smooth curve, any subsheaf of a locally free sheaf is locally free.

\begin{definition}
  \label{defn:semistability}
  A vector bundle $E$ on $X$ is
  \begin{itemize}
    \item \emph{semistable} if, for every nonzero subsheaf $E' \subset E$, we have $\mu(E') \leq \mu(E)$,
    \item \emph{stable} if it is semistable and the only subsheaf $E' \subset E$ with $\mu(E')= \mu(E)$ is $E' = E$,
    \item \emph{polystable} if it is isomorphic to a finite direct sum $\bigoplus_i E_i$, where each $E_i$ is stable with $\mu(E_i) = \mu(E)$.
  \end{itemize}
\end{definition}


Below are some basic properties of semistable vector bundles.

\begin{lemma}
  \label{lemma:semistable-tensor-line-bundle}
  For a line bundle $L$, a vector bundle $E$ is (semi)stable if and only if $E \otimes L$ is (semi)stable.
\end{lemma}

\begin{proof}
  Note that, for any vector bundle $E$, we have
  \[
    \deg(E \otimes L) = \deg(E) + \rk(E) \cdot \deg(L).
  \]
  Therefore, by definition of the slope,
  \[
    \mu(E \otimes L) = \mu(E) + \deg(L).
  \]
  Hence, for any subsheaf $E' \subset E$, we have
  \[
    \mu(E) - \mu(E') = \mu(E \otimes L) - \mu(E' \otimes L).
  \]
  By definition, $E$ is (semi)stable if and only if for all $E' \subset E$, the above quantity is positive (resp. non-negative), which holds if and only if $E \otimes L$ is (semi)stable.
\end{proof}

\begin{lemma}
  \label{lemma:no-morphisms}
  If $E$ and $F$ are semistable vector bundles with $\mu(E) > \mu(F)$, then $\Hom_X(E, F) = 0$.
\end{lemma}

\begin{proof}
  Suppose that on the contrary $E \to F$ is a nonzero map, and let $K$ and $I$ be its kernel and image respectively. Since $E$ is semistable, $\mu(K) \leq \mu(E)$. Hence,
  \begin{equation*}
    \begin{aligned}
      \mu(I) \rk(I)
      &= \mu(E) \rk(E) - \mu(K)\rk(K) \\
      &\geq \mu(E)\rk(E) - \mu(E) \rk(K) \\
      &= \mu(E) \rk(I)
    \end{aligned}
  \end{equation*}
  so we see that $\mu(I) \geq \mu(E) > \mu(F)$. But now $I$ cannot be a subsheaf of $F$ because $F$ is semistable.
\end{proof}
The following is proved similarly.
\begin{lemma}
  \label{lemma:stable-endomorphisms-field}
  Any nonzero endomorphism of a stable bundle~$E$ is invertible. In particular, as~$k$ is algebraically closed we have~$\operatorname{End}_X(E)\cong k$.
\end{lemma}

We will use the following lemma in the proof of Proposition \ref{proposition:boundedness-of-semistability} to establish boundedness.
\begin{lemma}
  \label{lemma:globally-generated}
  If $E$ is a semistable vector bundle on $X$ and $\mu(E) > 2g - 1$, then $E$ is globally generated and $\HH^1(X, E) = 0$.
\end{lemma}

\begin{proof}
  Let $x \in X(k)$ be a closed point. Let~$\mathcal{O}_x$ denote the skyscraper sheaf at~$x$. By taking the sheaf cohomology of the short exact sequence
  \[ 0 \to E \otimes \mathcal{O}_X(-x) \to E \to \mathcal{O}_x^{\oplus \rk(E)} \to 0, \]
  we see that it suffices to show that $\HH^1(X, E\otimes\mathcal{O}_X(-x)) = 0$. By Serre duality, this is equivalent to $\Hom_X(E\otimes\mathcal{O}_X(-x), \omega_X) = 0$. Now, $\mu(E\otimes\mathcal{O}_X(-x)) = \mu(E) - 1 > 2g - 2 = \mu(\omega_X)$ so the result follows by Lemma \ref{lemma:no-morphisms} since $\omega_X$ is semistable as a line bundle.
\end{proof}

\subsection{Jordan--H\"older and Harder--Narasimhan filtrations}
\label{subsection:filtrations}
Every semistable vector bundle $E$ on $X$ admits a \emph{Jordan--H\"older filtration}
\[
  0 = E_0 \subset E_1 \subset E_2 \subset \cdots \subset E_k = E
\]
where each subquotient $G_i = E_{i}/E_{i-1}$ is a \emph{stable} vector bundle with $\mu(G_i) = \mu(E)$.
This filtration is not unique, but the list of subquotients is unique by the following result. See \cite{potier} for a proof.
\begin{proposition}[Jordan--H\"older filtration]
  Let $E$ be a semistable vector bundle on $X$. Suppose that
  \[
    0 = E_0 \subset E_1 \subset E_2 \subset \cdots \subset E_k = E
  \]
  and
  \[
    0 = E_0' \subset E_1' \subset E_2'\subset \cdots \subset E_{\ell}' = E
  \]
  are two Jordan--H\"older filtrations with associated graded bundles $G_i = F_i/F_{i-1}$ and $G_i'= F_i'/F_{i-1}'$. We have $\ell = k$, and there exists a permutation $\sigma \in\Sym_k$ such that $G_i' \cong G_{\sigma(i)}$.
\end{proposition}
Hence the \emph{associated graded}~$\gr_E=\bigoplus_{i=1}^kG_i$ of a semistable bundle $E$ is well defined up to isomorphism. This allows us to make the following definition, where the S refers to Seshadri.
\begin{definition}
  Semistable vector bundles~$E$ and~$E'$ of the same rank and degree are called \emph{S-equivalent} if their associated graded vector bundles 
  are isomorphic.
\end{definition}

We can think of the Jordan--H\"older filtration as stating that semistable vector bundles are built out of stable vector bundles of the same slope. Meanwhile, the result below tells us that \emph{all} vector bundles are built out of semistable bundles in a \emph{unique} way. See again \cite{potier} for a proof.

\begin{proposition}[Harder--Narasimhan filtration]
  Let $F$ be a vector bundle on $X$. There exists a unique filtration of $F$ by subbundles
  \[
    0 = F_0 \subset F_1 \subset F_2 \subset \cdots \subset F_k = F
  \]
  such that $G_i = F_i/F_{i-1}$ is semistable for each $i$ and
  \[
    \mu(G_1) > \mu(G_2) > \cdots > \mu(G_k).
  \]
  This filtration is called the \emph{Harder--Narasimhan filtration}.
\end{proposition}

\begin{example}[$X = \mathbb{P}^1$]
  \label{example:P1}
  By a theorem of Birkhoff and Grothendieck, vector bundles on $X = \mathbb{P}^1$ split as direct sums of line bundles. Writing $E \cong \bigoplus_{i=1}^k \mathcal{O}_{\mathbb{P}^1}(d_i)^{\oplus m_i}$ with $d_1 > d_2 > \cdots > d_k$, the associated graded bundles of the Harder--Narasimhan filtration are $G_i = \mathcal{O}_{\mathbb{P}^1}(d_i)^{\oplus m_i}$ and their slopes are $\mu(G_i) = d_i$. In particular $E$ is semistable if and only if $k = 1$, and stable if and only if it is a line bundle.
\end{example}

\subsection{Openness of semistability}
\label{subsec:openness-of-semistability}
We can now show that semistability is an open property in families.
We will use the results from \cite[\spref{0DM1}]{stacks-project} on Quot, where we will denote~$\mathrm{Quot}_{\mathcal{E}/X/B}^{P(t)}$ the object from \cite[\spref{0DP6}]{stacks-project} for a numerical polynomial~$P(t)$, a morphism~$f\colon X\to B$ of schemes, and a quasicoherent sheaf~$\mathcal{E}$ on~$X$. 


\begin{proposition}[Openness of semistability]
  \label{prop:openness-of-semistability}
  Let $S$ be a $k$-scheme and $\mathcal{E}$ be a vector bundle of rank $r$ and relative degree $d$ on $X \times S \rightarrow S$. The points $p \in S$ for which $\mathcal{E}_p$ is semistable form an open subscheme.
\end{proposition}

\begin{proof}
    By standard approximation techniques we can assume that $S$ is of finite type over~$k$.  We may also assume that $S$ is connected.
    We will show that the locus of points $s \in S$ for which $\mathcal{E}_s$ is \emph{not semistable} is closed by expressing it as a finite union of images of proper morphisms. Because our family is flat, the Hilbert polynomial $\mathrm{P}_{\mathcal{E}_s}(t)$ is independent of $p \in S$.

  By definition, $\mathcal{E}_s$ is not semistable if and only if there exists some subsheaf $E' \hookrightarrow \mathcal{E}_s$ with $\mu(E') > d/r$. Considering the cokernel of $E' \hookrightarrow \mathcal{E}_s$, we see that $\mathcal{E}_s$ is not semistable if and only if $\mathcal{E}_s$ has a quotient with Hilbert polynomial $\mathrm{P}_{\mathcal{E}_s}(t) - \mathrm{P}_{E'}(t)$ for some $E'$ with $\mu(E') > d/r$. In other words, the set of $s \in S$ such that $\mathcal{E}_s$ is not semistable is the union of images of relative Quot schemes $\mathrm{Quot}^{Q(t)}_{\mathcal{E}/X\times S/S}$ for the projection $X \times S \to S$, where $Q(t)$ ranges over the set
  \begin{equation*}
    \left\{ Q(t) = \mathrm{P}_{\mathcal{E}_s}(t) - \mathrm{P}_{E'}(t) \mid \text{$\exists s \in S$, $E' \subset \mathcal{E}_s$ with $\mu(E') > d/r$} \right\}.
  \end{equation*}
  Relative Quot schemes are proper by \cite[\spref{0DPC}]{stacks-project}, so to show that the nonsemistable locus is closed, it suffices to show that the above set of polynomials is finite.

  The Hilbert polynomial $\mathrm{P}_{E'}(t)$ is determined by the rank and degree of $E'$. Being a subsheaf of $E$ there are only finitely many possibilities for the rank of $E'$. The requirement $\mu(E') > d/r$ puts a lower bound on the degree of $E'$. It remains to show that the degrees of all subsheaves of $\mathcal{E}_s$ as $s$ varies over $S$ are bounded above.

  To see this, choose a finite morphism~$X\to\mathbb{P}^1$  and consider the induced morphism $f\colon X \times S \rightarrow \mathbb{P}^1 \times S$. If $E' \subset \mathcal{E}_p$ is a subsheaf, then $f_* E' \subset f_* \mathcal{E}_p$ is a subsheaf.

  According to Example~\ref{example:P1}, vector bundles on $\mathbb{P}^1$ split as direct sums of line bundles. Only finitely many splitting types occur among the $f_*\mathcal{E}_s$ for $s \in S$, since otherwise $\hh^0(\mathbb{P}^1, f_*\mathcal{E}_s)$ would achieve infinitely many values, which is impossible by semicontinuity. Moreover, for each vector bundle on $\mathbb{P}^1$, the set of Euler characteristics of all its subsheaves is bounded above.

  Because $f$ is finite, $\chi(X,f_*E') = \chi(X,E')$. It follows that the Euler characteristics of all subsheaves $E' \subset \mathcal{E}_s$ are bounded above and, hence, the degrees of all subsheaves $E' \subset \mathcal{E}_s$ are bounded above.
\end{proof}

\subsection{Boundedness and irreducibility}
To obtain a finite-type moduli space, the objects parameterized by the moduli problem must be bounded in the following sense.

\begin{definition}[Boundedness]
  A collection of isomorphism classes of vector bundles on $X$ is called \emph{bounded} if there exists a scheme $S$ of finite type and a vector bundle $\mathcal{F}$ on $X \times S \to S$ such that every isomorphism class in the collection is represented by $\mathcal{F}_s$ for some $k$-point $s \in S$.
\end{definition}

When ranging over a bounded collection of isomorphism classes, the set of values of any discrete algebraic invariant must be finite. 
However, as soon as the rank is at least $2$, fixing the rank and degree is not sufficient to obtain a bounded family.

\begin{example}[An unbounded family] \label{notbounded}
Let $x \in X$ be a point, and consider the collection $\{F_n\}$ of rank $r \geq 2$, degree $d$ vector bundles
\[
  F_n = \mathcal{O}_X(-nx) \oplus \mathcal{O}_X((n+d)x) \oplus \mathcal{O}_X^{\oplus (r-2)}.
\]
The range of values of $\hh^1(X, F_n)$ is unbounded in this family. Hence, the collection of all isomorphism classes of vector bundles of fixed rank $r \geq 2$ and degree $d$ cannot be bounded.
\end{example}

One important source of bounded families comes from quotients of a fixed vector bundle.

\begin{example}[Boundedness of quotients]
  \label{quot}
  Fix a vector bundle $F$ on $X$ and a polynomial $P(t)\in\mathbb{Z}[t]$. The Quot scheme $\mathrm{Quot}^{P(t)}_{F/X/\Spec k}$ parameterizing quotients of $F$ with Hilbert polynomial $P(t)$ is of finite type by \cite[\spref{0DP9}]{stacks-project}. Therefore, the collection of isomorphism classes of vector bundles that appear as a quotient of $F$ with Hilbert polynomial $P(t)$ is bounded: restricting the universal quotient sheaf on $X \times \mathrm{Quot}^{P(t)}_{F/X/\Spec k}$ to the open locus where the quotient is locally free provides a family of vector bundles realizing each isomorphism class.
\end{example}

Example \ref{quot} helps us see that semistable bundles of fixed rank and degree are bounded.

\begin{proposition}[Boundedness of semistability]
  \label{proposition:boundedness-of-semistability}
  The collection of isomorphism classes of semistable vector bundles of rank $r$ and degree $d$ on $X$ is bounded.
\end{proposition}

\begin{proof}
  Let $x \in X$ be a $k$-point and let $n$ be an integer such that $d/r + n > 2g - 1$. For every semistable vector bundle $E$, the vector bundle $E\otimes\mathcal{O}_X(nx)$ is semistable by Lemma \ref{lemma:semistable-tensor-line-bundle} and has slope greater than $2g - 1$. Hence, by Lemma \ref{lemma:globally-generated}, the vector bundle $E\otimes\mathcal{O}_X(nx)$ is globally generated and $\hh^1(X, E\otimes\mathcal{O}_X(nx)) = 0$. In other words,
  $E\otimes\mathcal{O}_X(nx)$ is a quotient of $V \otimes \mathcal{O}_X$ where $V$ is a vector space of dimension $\hh^0(X, E\otimes\mathcal{O}_X(nx))$, and this dimension is independent of $E$. It follows that every semistable bundle $E$ of rank $r$ and degree $d$ arises as a quotient
\[ V \otimes \mathcal{O}_X(-nx) \rightarrow E.\]
Boundedness of this family now follows from Example \ref{quot} with $F = V \otimes \mathcal{O}_X(-nx)$ and $P(t) = d + r(t + 1 - g)$.
\end{proof}

\begin{proposition}[Irreducibility of the moduli stack] \label{prop:irreducibility}
  The moduli stack $\modulistack[\semistable]{X}{r}{d}$ is irreducible. The same is true for~$\modulistack[\semistable]{X}{r}{L}$.
\end{proposition}

\begin{proof}
  For a closed point $x \in X(k)$, Lemma \ref{lemma:semistable-tensor-line-bundle} says that $E$ is semistable if and only if $E\otimes\mathcal{O}_X(x)$ is semistable. This gives rise to an isomorphism $\modulistack[\semistable]{X}{r}{d} \to \mathcal{M}^{\semistable}_X(r,d+r)$. Therefore, it suffices to prove the claim when $d > r(2g - 1)$. In other words, we may assume that the slope is at least $2g - 1$, and so all semistable bundles of our fixed rank and degree are globally generated by Lemma \ref{lemma:globally-generated}.

  We will construct an irreducible variety $V^{\mathrm{ss}}$ that maps to $\modulistack[\semistable]{X}{r}{d}$ surjectively. Because $V^{\mathrm{ss}}$ is a scheme the morphism is automatically representable, and so we get an induced surjection on the level of topological spaces by \cite[\spref{04XI}]{stacks-project}. But then the topological space is irreducible, and hence so is the stack.  Our space $V^{\mathrm{ss}}$ will parameterize semistable extensions of a degree-$d$ line bundle by a rank-$(r - 1)$ trivial bundle. For a fixed line bundle $L$, the extensions
  \begin{equation} \label{ext}
  0 \rightarrow \mathcal{O}_X^{r-1} \rightarrow E \rightarrow L \rightarrow 0
  \end{equation}
  are parameterized by $\mathrm{Ext}_X^1(L, \mathcal{O}_X^{r-1}) = \HH^1(X, L^\vee)^{r-1}$. To parameterize such extensions of all degree $d$ line bundles, let $\mathcal{L}$ be a Poincar\'e line bundle on $X \times \Pic_X^d$ and let $\pi\colon X \times \Pic_X^d \to \Pic_X^d$ be the projection. By the theorem on cohomology and base change, the fiber of the vector bundle $V = (\mathrm{R}^1\pi_* \mathcal{L}^\vee)^{r-1}$ over $L \in \Pic_X^d$ is the space of extensions $\HH^1(X, L^\vee)^{r-1}$, and there is a universal extension on $X \times V$. By Proposition \ref{prop:openness-of-semistability}, the locus $V^{\mathrm{ss}}$ where the universal extension is semistable is an open subvariety, so we obtain a map $V^{\mathrm{ss}} \rightarrow \modulistack[\semistable]{X}{r}{d}$. Clearly $V$ is irreducible, hence so is the open subvariety $V^{\semistable}$.

  To show that this map is \emph{surjective}, we must show that every semistable bundle $E$ arises as an extension as in \eqref{ext}. By the first paragraph of the proof, $E$ is globally generated. For such an $E$, we claim there exists a subspace $\Gamma \subset \HH^0(X, E)$ of dimension $r - 1$, so that the evaluation map\begin{equation} \label{eval}\Gamma \otimes \mathcal{O}_X \rightarrow E
  \end{equation}
  has rank $r - 1$ at all points of $X$. To see this, we use a dimension count. At each point $x \in X$, we consider the map of vector spaces $\HH^0(X, E) \rightarrow E(x)$, where $E(x)$ denotes the fiber of $E$ at~$x$. We wish to select a $\Gamma$ in the Grassmannian $\mathrm{Gr}(r - 1, \HH^0(X, E))$ which does not meet the kernel nontrivially. Since the kernel has codimension $r$, the linear spaces meeting it form a codimension-$2$ Schubert variety in $\mathrm{Gr}(r - 1, \HH^0(X, E))$. Varying $x$ over $X$, we see that the collection of $\Gamma$'s where \eqref{eval} has rank less than $r - 1$ has codimension at least $1$. Hence, there exists some $\Gamma$ where \eqref{eval} has rank $r - 1$.
  For such $\Gamma$, we obtain an exact sequence
  \[0 \rightarrow \Gamma \otimes \mathcal{O}_X  \rightarrow E \rightarrow \det(E) \rightarrow 0.\]
  This shows that $E$ occurs as an extension of a degree $d$ line bundle by a rank-$(r - 1)$ trivial bundle. Hence, $E$ is in the image of $V^{\semistable} \to \modulistack[\semistable]{X}{r}{d}$.

The proof for~$\modulistack[\semistable]{X}{r}{L}$ is similar, using an open subscheme of $\Ext_X^1(L,\mathcal{O}_X^{r-1})$ surjecting onto the stack.
\end{proof}

We summarize the results of this section as follows.

\begin{theorem}
  \label{theorem:properties}
  The moduli stacks $\modulistack[\semistable]{X}{r}{d}$ and $\modulistack[\semistable]{X}{r}{L}$ are
  irreducible
  algebraic stacks,
  of finite type over $k$
  with affine diagonal.
  They are moreover smooth, and their Zariski tangent spaces are given in Proposition \ref{proposition:smooth}.
\end{theorem}

\begin{remark}
  \label{remark:existence-stable}
  One can show that if the genus of $X$ is at least $2$, then for any integer~$r > 0$ and line bundle $L \in \Pic_X$ there exist stable vector bundles of rank $r$ and determinant $L$; see, e.g.,~\cite[Lemma~4.3]{MR0242185}. Hence the open substacks~$\modulistack[\stable]{X}{r}{d}\subseteq\modulistack[\semistable]{X}{r}{d}$ and~$\modulistack[\stable]{X}{r}{L}\subseteq\modulistack[\semistable]{X}{r}{L}$ of \emph{stable} bundles are nonempty and dense.
\end{remark}

\section{Good moduli spaces}
\label{section:good-moduli-space}
To study the properties of an algebraic stack, we often want to find the best possible approximation as a scheme or algebraic space. For a Deligne--Mumford stack such a procedure is given by the Keel--Mori theorem, producing a coarse moduli space. For an algebraic stack the first author introduced the notion of a good moduli space in \cite{MR3237451} and in \cite{1812.01128v3} criteria to verify its existence were established. The goal of this section is to briefly discuss how $\modulistack[\semistable]{X}{r}{d}$ fits into this general framework.

The following example explains why the stack~$\modulistack[\semistable]{X}{r}{d}$ is not Deligne--Mumford and the Keel--Mori theorem does not apply.
\begin{example}
  \label{example:Aut}
  The automorphism group of any semistable vector bundle contains a canonical copy of $\mathbb{G}_{\mathrm{m}}$ acting by scalars, and moreover~$\Aut(F)=\mathbb{G}_{\mathrm{m}}$ when $F$ is a stable bundle by Lemma~\ref{lemma:stable-endomorphisms-field}. In particular, the stack~$\modulistack[\semistable]{X}{r}{d}$ is not Deligne--Mumford, as the automorphism groups of all points are infinite.

  A prototypical example of a strictly semistable vector bundle of rank~2 and degree~0 is the polystable bundle~$F=\mathcal{O}_X\oplus\mathcal{O}_X$. Now~$\Aut(F)=\mathrm{GL}_2$, so the dimension of the automorphism group can jump up when going from the stable locus to the strictly semistable locus.

  The automorphism groups do not have to be reductive either. If $F$ fits in a non-split exact sequence
  \[ 0 \to \mathcal{O}_X \xrightarrow{i} F \xrightarrow{q} \mathcal{O}_X \to 0, \]
  then the map $(a, b) \mapsto a \cdot \id_F + \, a b \cdot i \circ q$ gives an an isomorphism $\mathbb{G}_{\mathrm{m}}~\times~\mathbb{G}_{\mathrm{a}}~\cong \Aut(F)$. 
  Such extensions are parameterized by the projective space \[\mathbb{P}\mathrm{Ext}_X^1(\mathcal{O}_X,\mathcal{O}_X)\cong\mathbb{P}^{g-1}\].
\end{example}

\begin{example}
  \label{example:to-dm-or-not-to-dm}
  For~$\modulistack[\semistable]{X}{r}{L}$ as defined in \eqref{eq:cartesian square over det morphism}
  one has generically a finite stabilizer, because for a \emph{stable} vector bundle with fixed determinant the automorphisms need to preserve the isomorphism, and hence are isomorphic to~$\mu_r$.

  If instead we took the fiber product along the closed substack~$\mathrm{B}\mathbb{G}_{\mathrm{m}}\hookrightarrow\stPic_X^d$ then the resulting closed substack would have the same stabilizers as~$\modulistack[\semistable]{X}{r}{d}$, and in particular they would never be finite.

  Observe that, with the first definition of~$\modulistack[\semistable]{X}{r}{L}$, if~$\gcd(r,\deg L)=1$ then there are no strictly semistable vector bundles, and~$\modulistack[\semistable]{X}{r}{L}$ is Deligne--Mumford. In this case~$\modulistack[\semistable]{X}{r}{L}$ is a $\mu_r$-gerbe over its coarse moduli space. With the second definition the stack is rather a~$\mathbb{G}_{\mathrm{m}}$-gerbe over this coarse moduli space.
\end{example}

\begin{example}
  \label{example:not-separated}
  The moduli stack~$\modulistack[\semistable]{X}{r}{d}$ is \emph{not} separated, despite its many other nice properties. Indeed, if it were separated, its diagonal would be proper. But the diagonal is affine, hence it would be finite, and so the stabilizers would be finite as they are the fibers of the diagonal. This contradicts the previous example.
  Likewise,  the moduli stack $\modulistack[\semistable]{X}{r}{L}$ is \emph{not} separated if the strictly semistable locus is nonempty.
\end{example}

\begin{definition}
  Let~$f\colon\mathscr{X}\to X$ be a quasi-compact and quasi-separated morphism over $k$, where~$\mathscr{X}$ is an algebraic stack and~$X$ is an algebraic space. We say that~$f$ is a \emph{good moduli space} if
  \begin{enumerate}
    \item $f$ is cohomologically affine, i.e.,~$f_*\colon\operatorname{Qcoh}\mathscr{X}\to\operatorname{Qcoh}X$ is exact;
    \item the morphism $\mathcal{O}_X\to f_*\mathcal{O}_{\mathscr{X}}$ is an isomorphism.
  \end{enumerate}
\end{definition}

We now summarize some of the fundamental properties of good moduli spaces that will be relevant for us. See \cite{MR3237451} for a complete statement.
\begin{theorem}\label{theorem:gms-properties}
  Let $\mathscr{X}$ be an algebraic stack of finite type over $k$, and let $f\colon\mathscr{X} \to X$ be a good moduli space.
  \begin{enumerate}
      \item The map $f$ is surjective, universally closed, and initial for maps to algebraic spaces.
      \item The space $X$ is of finite type over $k$.
      \item The map $f$ identifies two $k$-points $x, y\colon \Spec k \to \mathscr{X}$ if and only if the closures of $\{x\}$ and $\{y\}$ in $|\mathscr{X}|$ intersect. In particular, $f$ induces a bijection on closed points.
      \item The pullback $f^*$ induces an equivalence of categories between vector bundles on $X$ and those vector bundles $\mathcal{F}$ on $\mathscr{X}$ such that the stabilizer of any closed point $x\colon \Spec k \to \mathscr{X}$ acts trivially on $x^* \mathcal{F}$.
  \end{enumerate}
\end{theorem}


The following existence theorem for good moduli spaces is a specialized version of \cite[Theorem~A]{1812.01128v3}, in the form in which we will apply it. The terms ``S-complete'' and ``$\Theta$-reductive'' in the statement refer to certain valuative criteria that we will discuss below.
\begin{theorem}
  \label{theorem:gms-existence}
  Let~$\mathscr{X}$ be an algebraic stack of finite type over~$k$ with affine diagonal. There exists a good moduli space~$\mathscr{X}\to X$ where~$X$ is a separated algebraic space over $k$ if and only if
  \begin{enumerate}
    \item $\mathscr{X}$ is S-complete, and
    \item $\mathscr{X}$ is $\Theta$-reductive.
  \end{enumerate}
  Moreover, $X$ is proper if and only if $\mathscr{X}$ is universally closed, or equivalently satisfies the existence part of the valuative criterion for properness.
\end{theorem}

\subsection{Langton's semistable reduction theorem}
We first discuss the last part of Theorem \ref{theorem:gms-existence}. Since the stacks $\modulistack[\semistable]{X}{r}{d}$ and $\modulistack[\semistable]{X}{r}{L}$ are quasicompact and have affine diagonal, hence are quasiseparated, it follows from \cite[Tags \splink{0CLW} and \splink{0CLX}]{stacks-project} that universal closedness is equivalent to the existence part of the valuative criterion for properness.


The valuative criterion states the following. If~$R$ is a discrete valuation ring~$R$ with residue field~$\kappa$ and fraction field~$K$, then for every 2-commutative diagram
\[
  \begin{tikzcd}
    \Spec K \arrow[r] \arrow[d, hook] & \mathscr{X} \arrow[d] \\
    \Spec R \arrow[r] & \Spec k
  \end{tikzcd}
\]
there exists a field extension~$K'/K$ and a discrete valuation ring~$R'\subseteq K'$ dominating~$R$, such that there exists a morphism~$\Spec R'\to\mathscr{X}$ making the diagram
\[
  \begin{tikzcd}
    \Spec K' \arrow[r] \arrow[d, hook] & \Spec K \arrow[r] \arrow[d, hook] & \mathscr{X} \arrow[d] \\
    \Spec R' \arrow[r] \arrow[rru, dashed] & \Spec R \arrow[r]  & \Spec k
  \end{tikzcd}
\]
2-commutative.

It turns out that for~$\modulistack[\semistable]{X}{r}{d}$ and $\modulistack[\semistable]{X}{r}{L}$ it is not necessary to consider field extensions of~$K$, by the following classical theorem of Langton \cite{MR0364255}.
\begin{theorem}\label{theorem:Langton}
  Let $E$ be a semistable vector bundle of rank $r$ and degree~$d$ on $X_K$. There exists a vector bundle $\mathcal{E}$ on $X_R$ such that $\mathcal{E}_K \cong E$ and $\mathcal{E}_\kappa$ is semistable. Thus, the stacks $\modulistack[\semistable]{X}{r}{d}$ and $\modulistack[\semistable]{X}{r}{L}$ are universally closed.
\end{theorem}
According to Example \ref{example:not-separated}, the stack $\modulistack[\semistable]{X}{r}{d}$ is not separated, and there will in general be many extensions of $E$ to $X_R$. However, for any two extensions, the central fibers will be \emph{S-equivalent}, as we will see in Theorem \ref{theorem:good-moduli-space}, justifying the notion of S-equivalence. In particular, over the stable locus $\modulistack[\stable]{X}{r}{d}$ any two extensions are isomorphic, although not uniquely.

\subsection{$\Theta$-reductivity and S-completeness}
We now turn to the more contemporary ingredients of Theorem \ref{theorem:gms-existence}. Let~$R$ denote a discrete valuation ring with field of fractions~$K$ and residue field~$k$, and let~$\pi$ be a uniformizer. We define the quotient stacks
\begin{itemize}
  \item $\Theta_R=[\Spec R[t]/\mathbb{G}_{\mathrm{m}}]$ where the action has weight~$-1$,
  \item $\overline{\mathrm{ST}}_R\colonequals[(\Spec R[s,t]/(st-\pi))/\mathbb{G}_{\mathrm{m}}]$ where the action has weight~$1$ on~$s$ and~$-1$ on~$t$.
\end{itemize}
Both $\Theta$-reductivity and S-completeness are lifting criteria for \emph{codimension-2} punctures in $\Theta_R$ and $\overline{\rm ST}_R$. The notion of $\Theta$-reductivity is encoded by the lifting criterion in the left diagram, whereas S-completeness is the lifting criterion in the right diagram:
\begin{equation}
  \begin{tikzcd}
    \Theta_R\setminus\{0\} \arrow[r] \arrow[d, hook] & \mathscr{X} \arrow[d] \\
    \Theta_R \arrow[r] \arrow[ru, dashed, "\exists!"] & \Spec k
  \end{tikzcd}
  \quad
  \begin{tikzcd}
    \overline{\mathrm{ST}}_R\setminus\{0\} \arrow[r] \arrow[d, hook] & \mathscr{X} \arrow[d] \\
    \overline{\mathrm{ST}}_R \arrow[r] \arrow[ru, dashed, "\exists!"] & \Spec k
  \end{tikzcd}
\end{equation}
The notion of $\Theta$-reductivity was introduced in \cite{1411.0627v4} and that of S-completeness in \cite[Section~3.5]{1812.01128v3}. If~$\mathscr{X}$ is an algebraic stack with affine diagonal, then in either criterion, a lift is automatically unique by \cite[Proposition~3.17]{1812.01128v3}.

Each of the two lifting criteria encodes the extension of a certain type of filtration over the puncture, which we will now describe.

\subsubsection{$\Theta$-reductivity}
We let $0 \in \Theta_R$ be the unique closed point defined by the vanishing of $t$ and the uniformizer $\pi \in R$.  Observe that $\Theta_R \setminus 0 = \Theta_K \cup_{\Spec K} \Spec R$.
We have the following two cartesian diagrams which describe the geometry of $\Theta_R$:
\begin{equation}
  \label{E:ThetaR-schematic}
  \begin{tikzcd}
    & \Spec R \arrow[rd, hook, "t\neq0"] & & \mathrm{B}_R\mathbb{G}_{\mathrm{m}} \arrow[ld, hook', swap, "t=0"] \\
    \Spec K \arrow[ru, hook] \arrow[rd, hook] & & \Theta_R & & \mathrm{B}_k\mathbb{G}_{\mathrm{m}} \arrow[ld, hook'] \arrow[lu, hook'] \\
    & \Theta_K \arrow[ru, "\pi\neq0", hook] & & \Theta_k \arrow[lu, "\pi=0", swap, hook']
  \end{tikzcd}
\end{equation}
Here the maps on the left are open embeddings and on the right are closed embeddings.



A morphism $\rho\colon\Theta \to \underline{{\rm Coh}}_X$ corresponds to the data of a coherent sheaf $E$ on $X$ and a filtration $0=E_0 \subset E_1 \subset \cdots \subset E_n = E$.  Under this correspondence, $\rho(1) = E$ and $\rho(0) = \bigoplus_i E_{i+1}/E_i$ is the associated graded.  The map $\rho\colon \Theta \to \underline{{\rm Coh}}_X$ factors through $\modulistack{X}{r}{d}$ (resp. $ \modulistack[\semistable]{X}{r}{d}$) if and only if $E$ is a vector bundle (resp.~semistable vector bundle) of rank $r$ and degree $d$ and each factor $E_{i+1}/E_i$ is a vector bundle (resp.~semistable vector bundle of slope $\mu = d/r$).
There are analogous descriptions for morphisms from  $\Theta_A$ for a $k$-algebra $A$.

Under this correspondence, the $\Theta$-reductivity of these stacks translates into a valuative criterion for filtrations.  For $\modulistack[\semistable]{X}{r}{d}$, $\Theta$-reductivity translates into the following: for every discrete valuation ring $R$ with fraction field $K$ and semistable vector bundle $E$ on $X_R$ of rank $r$ and degree $d$, any filtration $0=G_0 \subset G_1 \subset \cdots \subset G_n = E_K$ of the generic fiber where each factor $G_{i+1}/G_i$ is semistable of slope $\mu = d/r$ extends to a filtration $0=E_0 \subset E_1 \subset \cdots \subset E_n = E$ where each factor $E_{i+1}/E_i$ is a family of semistable vector bundles on $X$ over $R$. Note that the factors $G_{i+1}/G_i$ are semistable of slope $d/r$ for all $i$ if and only if $G_i$ is semistable of slope $d/r$ for all $i$ if and only if $\mu(G_i)  = d/r$ for all $i$.

\begin{proposition}
  \label{proposition:Theta-reductive}
  The stacks $\underline{{\rm Coh}}_X$, $\modulistack[\semistable]{X}{r}{d}$, and $\modulistack[\semistable]{X}{r}{L}$ are $\Theta$-reductive.
\end{proposition}

\begin{proof}
  First consider $\underline{{\rm Coh}}_X$. Let $E$ be an $R$-flat coherent sheaf on $X_R$ and let $0=G_0 \subset G_1 \subset \cdots \subset G_n = E_K$ be a filtration of the generic fiber. By descending induction starting with $j = n$, suppose that we have constructed a filtration $E_j \subset E_{j+1} \subset \cdots \subset E_n = E$ extending $G_\bullet$. Since the relative Quot scheme of quotients of $E_j$ with the same Hilbert polynomial as $G_j/G_{j-1}$ is proper over $R$, there is a unique subsheaf $E_{j-1}\subset E_j$ with $(E_{j-1})_K = G_{j-1}$ and $E_j/E_{j-1}$ flat over $R$. This gives the next step in the filtration $E_\bullet$.

  To see that $\modulistack[\semistable]{X}{r}{d}$ is $\Theta$-reductive, suppose that $E$ is a family of semistable vector bundles on $X_R$ and $G_\bullet$ is a filtration of $E_K$ such that the sheaves $G_i$ are semistable vector bundles with $\mu(G_i) = d/r$. By the previous paragraph, $G_\bullet$ extends to $E_\bullet$ over $R$, and we must show that in the sheaves $(E_i)_k$ are semistable as well. Since each $E_i$ is flat over $R$, we have $\mu((E_i)_k) = \mu(G_i) = d/r$, and so $(E_i)_k$ is a subsheaf of the semistable vector bundle $E_k$ of the same slope, hence a semistable vector bundle itself. The same argument shows that $\modulistack[\semistable]{X}{r}{L}$ is $\Theta$-reductive.


\end{proof}

\subsubsection{S-completeness}
The stack $\overline{\mathrm{ST}}_R$ can be viewed as a local model of the quotient $[\mathbb{A}^2/\mathbb{G}_{\mathrm{m}}]$ where $\mathbb{A}^2$ has coordinates $s$ and $t$ with weights $1$ and $-1$; indeed, $\overline{\mathrm{ST}}_R$ is the base change of the good moduli space $[\mathbb{A}^2/\mathbb{G}_m] \to \Spec k[st]$ along $\Spec R \to  \Spec k[st]$ defined by $st \mapsto \pi$.

We let $0 \in \overline{\mathrm{ST}}_R$ be the unique closed point defined by the vanishing of $s$ and $t$.  Observe that $\overline{\mathrm{ST}}_R \setminus 0$ is the non-separated union $\Spec R \cup_{\Spec K} \Spec R$.
We have the following two cartesian diagrams which describe the geometry of $\overline{\mathrm{ST}}_R$:
\begin{equation}
  \label{E:ST-schematic}
  \begin{tikzcd}
    & \Spec R \arrow[rd, hook, "s\neq0"] & & \Theta_k \arrow[ld, "s=0", hook', swap] \\
    \Spec K \arrow[ru, hook] \arrow[rd, hook] & & \overline{\rm ST}_R & & \mathrm{B}_k\mathbb{G}_{\mathrm{m}} \arrow[lu, hook'] \arrow[ld, hook'] \\
    & \Spec R \arrow[ru, hook, "t\neq0"] & & \Theta_k \arrow[lu, "t=0", hook', swap]
  \end{tikzcd}.
\end{equation}
Here the maps on the left are open embeddings and those on the right are closed embeddings.




\begin{remark} \label{R:reductive}
  If $G$ is a linear algebraic group over $k$, then $\mathrm{B}G$ is S-complete if and only if $G$ is reductive (see \cite[Proposition~3.45 and Remark~3.46]{1812.01128v3}).  Moreover, as S-completeness is preserved under closed substacks, it follows that every closed point (corresponding to a polystable object) in an S-complete algebraic stack with affine diagonal has a \emph{reductive stabilizer}.

  We will see in Lemma \ref{lemma:topology-of-semistable-stack} below that the closed points of $\modulistack[\semistable]{X}{r}{d}$ correspond to polystable vector bundles, so although some $k$-points of $\modulistack[\semistable]{X}{r}{d}$ may have nonreductive stabilizers as in Example \ref{example:Aut}, such points necessarily correspond to non-polystable bundles.
\end{remark}

\begin{proposition}
  \label{proposition:S-complete}
  The stacks $\underline{{\rm Coh}}_X$, $\modulistack[\semistable]{X}{r}{d}$, and $\modulistack[\semistable]{X}{r}{L}$ are S-complete.
\end{proposition}

We omit the proof of this proposition, which can be proven along the same lines as Proposition \ref{proposition:Theta-reductive} using the interpretation of morphisms from $\overline{\rm ST}_R$ to $\modulistack[\semistable]{X}{r}{d}$ from \cite[Remark~3.36]{1812.01128v3}. 


We have summarized in Table \ref{table:properties} the properties of $\modulistack[\semistable]{X}{r}{d}$, $\modulistack{X}{r}{d}$, and $\underline{\rm Coh}_X$ discussed in the first three sections.

\begin{table}[t]
  \begin{adjustbox}{center}
    \small
    \begin{tabular}{lccc}
      \toprule
      Property                                                                                                                     & $\modulistack[\semistable]{X}{r}{d}$                       & $\modulistack{X}{r}{d}$                       & $\underline{\rm Coh}_X$ \\\midrule
      affine diagonal                                                                                                              & yes                                                        & yes,                                          & yes, \\
                                                                                                                                   &                                                            & Proposition \ref{proposition:affine-diagonal} & \cite[\spref{0DLY}]{stacks-project} \\
      \addlinespace
      locally of finite type                                                                                                       & yes                                                        & yes,                                          & yes,  \\
                                                                                                                                   &                                                            & Proposition \ref{proposition:affine-diagonal} & \cite[\spref{0DLZ}]{stacks-project} \\
      \addlinespace
      quasicompact                                                                                                                 & yes,                                                       & no,                                           & no \\
                                                                                                                                   & Proposition \ref{proposition:boundedness-of-semistability} & Example \ref{notbounded} \\
      \addlinespace
      $\Theta$-reductive                                                                                                           & yes,                                                       & no                                            & yes, \\
                                                                                                                                   & Proposition \ref{proposition:Theta-reductive}              &                                               & Proposition \ref{proposition:Theta-reductive} \\
      \addlinespace
      S-complete                                                                                                                   & yes,                                                       & no                                            & yes, \\
                                                                                                                                   & Proposition \ref{proposition:S-complete}                   &                                               & Proposition \ref{proposition:S-complete} \\
      \addlinespace
      existence part                                                                                                               & yes                                                        & yes                                           & yes \\
      of valuative criterion \\
      \addlinespace
      separated                                                                                                                    & no,                                                        & no                                            & no \\
                                                                                                                                   & Example \ref{example:not-separated} \\
      \addlinespace
      existence of                                                                                                                 & yes,                                                       & no                                            & no \\
      good moduli space                                                                                                            & Theorem \ref{theorem:good-moduli-space} \\
      \bottomrule
    \end{tabular}
  \end{adjustbox}
  \caption{Properties of the stacks $\modulistack[\semistable]{X}{r}{d}\subset\modulistack{X}{r}{d}\subset\underline{\rm Coh}_X$}
  \label{table:properties}
\end{table}

\subsection{Existence of good moduli spaces}
In this subsection we will establish the existence and basic properties of good moduli spaces for the stacks $\modulistack[\semistable]{X}{r}{d}$ and $\modulistack[\semistable]{X}{r}{L}$, but before doing so we make an observation regarding their topology.
\begin{lemma} \label{lemma:topology-of-semistable-stack}
  \
  \begin{enumerate}
      \item If $E$ is a semistable vector bundle, then the corresponding $k$-point $[E] \in \modulistack[\semistable]{X}{r}{d}(k)$ contains the point $[\gr_E]$ in its closure, where $\gr_E$ is the graded bundle associated to a Jordan--H\"older filtration.
      \item A point $[E] \in \modulistack[\semistable]{X}{r}{d}(k)$ is closed if and only if $E$ is polystable.
  \end{enumerate}
  The same holds for the stack $\modulistack[\semistable]{X}{r}{L}$.
\end{lemma}
\begin{proof}
  \begin{enumerate}
    \item If $E$ is semistable but not stable, there exists a non-split extension~$0\to E'\to E\to E''\to 0$ of semistable vector bundles of the same slope. Let $\mathcal{E}$ be the universal family over the affine line in $\Ext^1_X(E'',E')$ spanned by this extension, so that $\mathcal{E}$ is a family of semistable vector bundles on $X$ parameterized by $\mathbb{A}^1$ such that $\mathcal{E}_t\cong E$ if $t\neq 0$ and $\mathcal{E}_0\cong E'\oplus E''$. The resulting map
    \[
      [\mathcal{E}]\colon\mathbb{A}^1\to\modulistack[\semistable]{X}{r}{d}
    \]
    satisfies $t\mapsto [E]$ if $t\neq 0$ and $0\mapsto [E'\oplus E'']$. It follows that $[E'\oplus E'']$ is contained in the closure of $[E]$. Iterating this construction for $E'$ and $E''$ shows that $[\gr_E]$ is in the closure of $[E]$.

    \item For a contradiction, suppose $E$ is a polystable vector bundle such that $[E]$ is not closed, and let $[F]$ be in its closure. By (i), $[\gr_F]$ is in the closure of $[F]$ and, since no two points can be in the closure of each other, we must have $E \not\cong \gr_F$.

      On the other hand, if $E_i$ is stable with the same slope as $E$, then $E_i$ appears as a direct summand of $E$ with multiplicity $\mathrm{hom}_X(E_i,E)=\dim_k\Hom_X(E_i,E)$ and similarly for $\gr_F$.

      For any $E_i$, the function $\mathrm{hom}_X(E_i,-)$ is upper semicontinuous in the second variable, since $[\gr_F]$ is in the closure of $[E]$, so we have $\mathrm{hom}_X(E_i,E)\leq\mathrm{hom}_X(E_i,\gr_F)$. This means that any stable summand of $E$ appears in $\gr_F$ with at least the same multiplicity. But $E$ and $F$ have the same rank, so we must have $E \cong \gr_F$, a contradiction. Thus, $[E]$ is closed.
  \end{enumerate}
\end{proof}

\begin{theorem}\label{theorem:good-moduli-space}
  \ \vspace{-.2cm}
  \begin{enumerate}
    \item There exist good moduli spaces
      \begin{equation*}
        \begin{gathered}
          \modulistack[\semistable]{X}{r}{d} \to \modulispace[\semistable]{X}{r}{d} \\
          \modulistack[\semistable]{X}{r}{L} \to \modulispace[\semistable]{X}{r}{L}.
        \end{gathered}
      \end{equation*}

    \item These good moduli space maps induce bijections between the $k$-points of the good moduli space and S-equivalence classes of semistable sheaves.

    \item The good moduli spaces $\modulispace[\semistable]{X}{r}{d}$ and $\modulispace[\semistable]{X}{r}{L}$ are irreducible, proper algebraic spaces of dimensions
      \begin{equation*}
\begin{aligned}
  \dim \modulispace[\semistable]{X}{r}{d} &= r^2(g-1) + 1, \\
    \dim \modulispace[\semistable]{X}{r}{L} &= (r^2-1)(g-1).
\end{aligned}
\end{equation*}
  \end{enumerate}
\end{theorem}
\begin{proof}\
  \begin{enumerate}
      \item This follows from Propositions \ref{proposition:Theta-reductive} and \ref{proposition:S-complete} and the first part of Theorem \ref{theorem:gms-existence}.

      \item By Theorem \ref{theorem:gms-properties}(iii), two $k$-points $[E], [E'] \in \modulistack[\semistable]{X}{r}{d}$ map to the same point in $\modulispace[\semistable]{X}{r}{d}$ if and only if the closures of $\{[E]\}$ and $\{[E']\}$ in $\modulistack[\semistable]{X}{r}{d}$ intersect. On the one hand, if $E$ is any semistable vector bundle, then by Lemma \ref{lemma:topology-of-semistable-stack}(i) $[E]$ contains $[\gr_E]$ in its closure, so both points map to the same point in $\modulispace[\semistable]{X}{r}{d}$. On the other hand, if $E$ and $E'$ are polystable and nonisomorphic, then by Lemma \ref{lemma:topology-of-semistable-stack}(ii), the corresponding points in $\modulistack[\semistable]{X}{r}{d}$ are closed and distinct, hence map to distinct points in $\modulispace[\semistable]{X}{r}{d}$.

      \item The stacks are irreducible by Proposition \ref{prop:irreducibility} and the good moduli space maps are surjective by Theorem \ref{theorem:gms-properties}(i), so the good moduli spaces are irreducible as well. They are proper by Proposition \ref{theorem:Langton} and the second part of Theorem \ref{theorem:gms-existence}.

      To compute the dimensions, one can show using (ii) and Lemma \ref{lemma:stable-endomorphisms-field} that over the stable locus $\modulistack[\stable]{X}{r}{d} \subs \modulistack[\semistable]{X}{r}{d}$ the good moduli space map is a $\mathbb{G}_{\mathrm{m}}$-gerbe over its image $\modulispace[\stable]{X}{r}{d}$, and similarly for $\modulistack[\semistable]{X}{r}{L}$. Since the stable loci are moreover nonempty by Remark \ref{remark:existence-stable}, the dimensions now follow from Proposition \ref{proposition:smooth} and Remark \ref{remark:dimension}.\qedhere
  \end{enumerate}
\end{proof}

\section{Determinantal line bundles}
\label{section:determinantal}
Let $X$ be a smooth, projective and connected curve over $k$. For an algebraic stack $\mathscr{S}$ over $k$ we consider the diagram
\[
  \begin{tikzcd}
    & X\times \mathscr{S} \arrow[ld, swap, "q"] \arrow[rd, "p"] \\
    X & & \mathscr{S}
  \end{tikzcd}
\]
In this setting we can freely apply cohomology and base change arguments, because~$p$ is representable by schemes. If $\mathcal{E}$ is a vector bundle on $X \times \mathscr{S}$, then the derived direct image $\mathbf{R}p_* \mathcal{E}$ is a perfect complex on $\mathscr{S}$ with amplitude in $[0,1]$. In fact, something stronger is true. We will use the construction of \cite[Proposition~2.1.10]{MR2665168}. Let $\mathcal{E}$ be a vector bundle on $X\times \modulistack{X}{r}{d}$. Then there exists a short exact sequence
\[
    0\to\mathcal{E}^{-1}\to\mathcal{E}^0\to\mathcal{E}\to 0
\]
of vector bundles such that $\mathrm{R}^0p_*\mathcal{E}^j=0$ and $K^j\colonequals\mathrm{R}^1\mathcal{E}^{j+1}$ is locally free for $j=0,1$. In particular, we have a quasi-isomorphism $K^{\bullet}\colonequals[K^0\to K^1]\to\mathbf{R}p_*\mathcal{E}$.
We make the following definition.
\begin{definition}[Determinants and sections]
  \label{defn:determinants-and-sections}
  If $\mathcal{E}$ is a vector bundle on $X \times \mathscr{S}$ and $[K^0\to K^1]$ is a two-term complex of locally free sheaves such that $\mathbf{R}p_* \mathcal{E} \sim [K^0 \to K^1]$, we define the line bundle
  \[
    \det \mathbf{R}p_* \mathcal{E}\colonequals \det(K^0) \otimes \det(K^1)^{\vee}.
  \]
  If $\rk (\mathbf{R}p_* \mathcal{E}) = 0$, then $\rk K^0 = \rk K^1$ and the dual $(\det \mathbf{R}p_* \mathcal{E})^{\vee}$ is equipped with a section, locally given by the determinant of the map $K^0 \to K^1$.
\end{definition}
As described in \cite[\spref{0FJI}]{stacks-project}, the definition of $\det\mathbf{R}p_*\mathcal{E}$ is independent of the choice of quasi-isomorphism with a two-term complex.



We will apply this construction in the case $\mathscr{S} = \modulistack{X}{r}{d}$ and $\mathcal{E} = \universalbundle \otimes q^*V$, where $\universalbundle$ is the universal vector bundle on $X \times \modulistack{X}{r}{d}$ and $V$ is a vector bundle on $X$. The universal vector bundle exists tautologically because we are working with the moduli stack. Consider the diagram
\[
  \begin{tikzcd}
    & X \times \modulistack{X}{r}{d} \arrow[ld, swap, "q"] \arrow[rd, "p"] \\
    X & & \modulistack{X}{r}{d}.
  \end{tikzcd}
\]

\begin{proposition}[Determinant is multiplicative]
  \label{proposition:determinant-mult}
 If
\[
  0 \to \mathcal{E}' \to \mathcal{E} \to \mathcal{E}'' \to 0,
\]
is an exact sequence of vector bundles on $X \times \modulistack{X}{r}{d}$, then
\[
\det \mathbf{R}p_* \mathcal{E} = (\det \mathbf{R}p_* \mathcal{E}') \otimes (\det \mathbf{R}p_* \mathcal{E}'').
\]
\end{proposition}

\begin{proof}
As before, we can find a two-term complex $\mathcal{E}^{\bullet}=[\mathcal{E}^{-1}\to\mathcal{E}^0]$ supported in degrees $[-1,0]$ and a quasi-isomorphism $\mathcal{E}^{\bullet}\to \mathcal{E}$ with the properties described above. We choose $\mathcal{E}^{\prime\bullet}$ and $\mathcal{E}^{\prime\prime\bullet}$ similarly. In fact, it follows from the construction in the proof of \cite[Proposition~2.1.10]{MR2665168} that we may choose these resolutions compatibly, so that we have a short exact sequence
\[
    0\to\mathcal{E}^{\prime\bullet}\to\mathcal{E}^{\bullet}\to\mathcal{E}^{\prime\prime\bullet}\to 0
\]
of complexes, compatible with the quasi-isomorphisms to the members in our original exact sequence. Taking cohomology, we find a short exact sequence
\[
    0\to K^{\prime\bullet}\to K^{\bullet}\to K^{\prime\prime\bullet}\to 0
\]
of complexes of locally free sheaves on $\modulistack{X}{r}{d}$. The result follows from the multiplicativity of the determinant in short exact sequences of locally free sheaves.
\end{proof}

\begin{definition}[Determinantal line bundles and sections]
  \label{definition:determinantal}
  For a vector bundle $V$ on $X$, we define the \emph{determinantal line bundle}
  \[
    \mathcal{L}_V\colonequals(\det \mathbf{R}p_*(q^*V \otimes \universalbundle))^{\vee}
  \]
  on $\modulistack{X}{r}{d}$ associated to~$V$. If $\chi(X,V \otimes E) = 0$ for all $[E] \in \modulistack{X}{r}{d}$, or equivalently by the Riemann--Roch theorem
  \begin{equation}
    \label{equation:riemann-roch-condition}
    d \rk(V) + r \deg(V) + r \rk(V) (1-g) = 0,
  \end{equation}
  then the rank of $\mathbf{R}p_*(q^*V \otimes \universalbundle)$ is zero and we define the section $s_V \in \Gamma(\modulistack{X}{r}{d}, \mathcal{L}_V)$ as in Definition \ref{defn:determinants-and-sections}.
\end{definition}

\begin{remark}
  Since $\mathbf{R}p_*(q^*V \otimes \universalbundle)$ is perfect, its construction commutes with base change.  In particular, its restriction to a $k$-point $[E] \in \modulistack{X}{r}{d}$ is identified with the two-term complex $\mathbf{R}\Gamma(X, E \otimes V)$. If moreover $\chi(X, V \otimes E) = \hh^0(X, V \otimes E) - \hh^1(X, V \otimes E) = 0$, we see that indeed $\rk(\mathbf{R}p_*(q^*V \otimes \universalbundle)) = 0$.
\end{remark}

In order to state some basic properties of the construction in Definition \ref{definition:determinantal}, we recall that, by the universal property of the Picard stack $\stPic^d_X$, the determinant $\det(\universalbundle)$ of the universal vector bundle on $X \times \modulistack{X}{r}{d}$ induces a morphism
\[ \det\colon\modulistack{X}{r}{d} \to \stPic_X^d \]
such that $(\det \times \id_X)^*\mathscr{P} \cong \det(\universalbundle)$, where $\mathscr{P}$ is the Poincar\'e bundle on $\stPic_X^d \times X$.

\begin{proposition}[Properties of the determinantal line bundle]
  \label{proposition:donaldson-morphism}
  The following hold:
\begin{enumerate}
    \item The assignment $V \mapsto \mathcal{L}_V$ induces a group homomorphism
  \begin{equation}
    \label{equation:donaldson}
    \Kzero(X) \to \Pic(\modulistack{X}{r}{d}).
  \end{equation}
  Consequently, the isomorphism class of the line bundle $\mathcal{L}_V$ depends only on $\rk(V)$ and $\det(V)$.

  \item If $V$ and $W$ are vector bundles with equal rank and degree, then there exists a line bundle $\mathcal{N}$ on $\stPic_X^d$ such that
    \[ \mathcal{L}_W \cong \mathcal{L}_V \otimes \operatorname{det}^*\mathcal{N}. \]
\end{enumerate}

\end{proposition}
\begin{proof}
        An exact sequence $0 \to V_1 \to V_2 \to V_3 \to 0$ of vector bundles on $X$ induces an exact sequence of vector bundles
        \[ 0 \to q^* V_1 \otimes \universalbundle \to q^* V_2 \otimes \universalbundle \to q^* V_3 \otimes \universalbundle \to 0, \]
        which by Proposition \ref{proposition:determinant-mult} induces an isomorphism
        \[ \mathcal{L}_{V_2} \cong \mathcal{L}_{V_1} \otimes \mathcal{L}_{V_3}. \]
        Thus, we have a group homomorphism $\Kzero(X) \to \Pic(\modulistack{X}{r}{d})$. The second statement follows from the isomorphism $\Kzero(X) \cong \mathbb{Z} \otimes \Pic(X)$ given by $V \mapsto (\rk(V), \det(V))$.
        This proves (i).

        By (i), the isomorphism type of $\mathcal{L}_V$ depends only on the rank $r_V$ and determinant of $V$, so we may assume that $V = \mathcal{O}_X^{\oplus r_V - 1} \oplus \mathcal{O}_X(D)$ for some divisor $D$ on $X$, and similarly for $W$. Moreover, writing $D = D_1 - D_2$ as a difference of effective divisors and using standard exact sequences, we see that $[\mathcal{O}_X(D)] = [\mathcal{O}_X] + [\mathcal{O}_{D_1}] - [\mathcal{O}_{D_2}]$ in $\Kzero(X)$. This implies that the classes of $V$ and $W$ in $\Kzero(X)$ differ only by the class of a divisor of degree $0$. Thus, by the additivity of the determinantal construction, it suffices to prove that
          \[ \det \mathbf{R} p_*(\universalbundle \otimes q^*\mathcal{O}_{x}) \cong \operatorname{det}^*\mathcal{N}' \]
        for some line bundle $\mathcal{N}'$ on $\stPic_X^d$, where $x \in X$ is a closed point. If we view $\mathcal{E}_x = \universalbundle|_{\modulistack[\semistable]{X}{r}{d} \times \{x\}}$ and $\mathcal{P}_x = \mathcal{P}|_{\stPic_X^d \times \{x\}}$ as sheaves on $\modulistack[\semistable]{X}{r}{d}$ and $\stPic_X^d$ respectively, then we have
        \[ \det \mathbf{R} p_*(\universalbundle \otimes q^*\mathcal{O}_x) = \det \mathcal{E}_x = \operatorname{det}^* \mathcal{P}_x. \]
        This proves (ii).

\end{proof}

\begin{definition}
  Let~$E$ be a vector bundle on~$X$. We say that~$E$ is \emph{cohomology-free} if~$\hh^0(X,E)=\hh^1(X,E)=0$.
\end{definition}


The following proposition relates cohomology-freeness to the properties of sections of determinantal line bundles. It will be an essential tool in the ampleness proof.

\begin{proposition}
  \label{proposition:determinantal-line-bundle-properties}
  If $\chi(X,E \otimes V) = 0$ for all $[E] \in \modulistack{X}{r}{d}$, then the following are equivalent:
  \begin{itemize}
    \item the section $s_V \in \Gamma(\modulistack{X}{r}{d}, \mathcal{L}_V)$ is nonzero at a vector bundle $[E] \in \modulistack{X}{r}{d}$;
    \item $E \otimes V$ is cohomology-free.
  \end{itemize}
\end{proposition}

\begin{proof}
  In the setup of Definition \ref{definition:determinantal} the morphism of line bundles $\det j\colon \det(K_0) \to \det(K_1)$ is nonzero at the point~$[E]\in\modulistack{X}{r}{d}$ if and only if the morphism of vector bundles~$j\colon K_0 \to K_1$ is an isomorphism at~$[E]$. Since the derived direct image sheaf is the cohomology of this complex, this occurs if and only if $\hh^0(X,E\otimes V) = \hh^1(X,E\otimes V) =0$. 
\end{proof}

\begin{remark}
  We emphasize that while $\mathcal{L}_V$ only depends on $\rk(V)$ and $\det(V)$, the section $s_V$ \emph{does} depend on $V$ itself.  We will leverage this fact to produce enough sections of $\mathcal{L}_V$ to establish ampleness. Notice also that, under the assumption $\chi(X,E \otimes V) = 0$, the vanishing of $\HH^0(X, E \otimes V)$ is equivalent to the vanishing of $\HH^1(X, E \otimes V)$.
\end{remark}

We now specialize the construction to the open substack~$\modulistack[\semistable]{X}{r}{d} \subs \modulistack{X}{r}{d}$. As explained in the introduction, our goal is to prove that the proper algebraic space~$\modulispace[\semistable]{X}{r}{d}$ obtained in Theorem~\ref{theorem:good-moduli-space} is actually a projective scheme. For this we need to produce an ample line bundle on it, and the determinantal line bundle we have constructed a priori lives on~$\modulistack[\semistable]{X}{r}{d}$.
To remedy this with the following. 

\begin{proposition}
  The determinantal line bundle~$\mathcal{L}_V$ of Proposition \ref{proposition:determinantal-line-bundle-properties} associated to a vector bundle~$V$ descends to~$\modulispace[\semistable]{X}{r}{d}$, that is, there exists a unique line bundle~$\mathrm{L}_V\in\Pic(\modulispace[\semistable]{X}{r}{d})$ such that~$\mathcal{L}_V\cong\phi^*\mathrm{L}_V$.
\end{proposition}

\begin{proof}
  By Theorem \ref{theorem:gms-properties}(iv), we must show that stabilizers of $\modulistack[\semistable]{X}{r}{d}$ act trivially on the fibers of $\mathcal{L}_V$. By Theorem \ref{theorem:good-moduli-space}(ii), the closed points of $\modulistack[\semistable]{X}{r}{d}$ correspond to polystable bundles, and for a polystable bundle $E = \bigoplus_{j=1}^n E_j^{\oplus m_j}$, where the $E_i$ are pairwise nonisomorphic, we have $\Aut(E) \cong \GL_{m_1} \times \cdots \times \GL_{m_n}$. The fiber of $\mathcal{L}_V|_{[E]}$ is identified with
  \[ \det \mathbf{R}\Gamma(X, E \otimes V) = \prod_{i=1}^n (\det \HH^i(X, E \otimes V))^{\otimes (-1)^i} . \]
  An element $(g_1,\ldots,g_n) \in \Aut(E)$ acts on
  \[ \det\HH^i(X, E \otimes V) \cong \bigotimes_{j=1}^n \left(\det\HH^i(X, E_j \otimes V)\right)^{\otimes m_j} \]
  by multiplication with $\prod_{j=1}^n \det(g_1)^{\dim\HH^i(X, E_j \otimes V)}$, and thus on $\mathcal{L}_V|_{[E]}$ by multiplication with $\prod_{j=1}^n \det(g_j)^{\chi(X, E_j \otimes V)}$. But each $E_j$ has slope equal to $d/r$, so by the assumption on $V$ we have $\chi(X, E_j \otimes V) = 0$, and so the action is trivial.
\end{proof}

\begin{corollary}
  For a line bundle $L$ of degree $d$ on $X$, the restriction of the determinantal line bundle $\mathrm{L}_V$ to $\modulispace[\semistable]{X}{r}{L}$ only depends on the rank and degree of $V$.
\end{corollary}
\begin{proof}
  We have commuting diagrams
  \begin{center}
  \begin{tikzcd}
    \modulistack[\semistable]{X}{r}{L} \arrow[r] \arrow[d] \arrow[rd, phantom, "\square"] & \modulistack[\semistable]{X}{r}{d} \arrow[d, "\det"]  \\
    \Spec k \arrow[r, "{[L]}"] & \stPic_X^d
  \end{tikzcd}
  and
  \begin{tikzcd}
    \modulistack[\semistable]{X}{r}{L} \arrow[r] \arrow[d] \arrow[rd, phantom] & \modulistack[\semistable]{X}{r}{d} \arrow[d] \\
    \modulispace[\semistable]{X}{r}{L} \arrow[r, hook] & \modulispace[\semistable]{X}{r}{d}.
  \end{tikzcd}
  \end{center}
  where in the second square, both vertical arrows are good moduli space maps. If $V$ and $W$ are vector bundles of equal rank and degree on $X$, both satisfying condition \eqref{equation:riemann-roch-condition}, then by Proposition \ref{proposition:determinantal-line-bundle-properties}(ii), there exists a line bundle $\mathcal{N}$ on $\stPic_X^d$ such that $\mathcal{L}_W \cong \mathcal{L}_V \otimes \operatorname{det}^*\mathcal{N}$. The left diagram shows that, restricting to $\modulistack[\semistable]{X}{r}{L}$, this isomorphism becomes $\mathcal{L}_W \cong \mathcal{L}_V$. The right diagram now shows that the restrictions of $\mathrm{L}_V$ and $\mathrm{L}_W$ to $\modulispace[\semistable]{X}{r}{L}$ become isomorphic after pulling back to $\modulistack[\semistable]{X}{r}{L}$, so the restrictions must be isomorphic to begin with, by Theorem \ref{theorem:gms-properties}(iv).
\end{proof}


Using geometric invariant theory, Dr\'ezet--Narasimhan gave in \cite[Th\'eor\`emes B, C]{MR0999313} the following description of the Picard groups of the good moduli spaces.
\begin{theorem}[Dr\'ezet--Narasimhan]
  There exist isomorphisms
  \begin{equation}
    \begin{aligned}
      \Pic(\modulispace[\semistable]{X}{r}{d})&\cong\Pic(\Pic_X^d)\oplus\mathbb{Z}, \\
      \Pic(\modulispace[\semistable]{X}{r}{L})&\cong\mathbb{Z}. \\
    \end{aligned}
  \end{equation}
  In the second line~$\mathbb{Z}$ is generated by the determinantal line bundle~$\mathrm{L}_V$, where~$V$ is chosen to be of minimal rank.
\end{theorem}
This explains how the isomorphism type of the line bundle $\mathrm{L}_V$ on the good moduli space depends on the invariants of the vector bundle $V$ satisfying $\chi(X,E\otimes V) = 0$, building upon Proposition \ref{proposition:donaldson-morphism}.

\section{Projectivity}
\label{section:projectivity}
In this section we will prove the following result.

\begin{theorem}
  \label{theorem:projectivity}
  Let~$X$ be a smooth projective curve of genus~$g\geq 2$.
  The good moduli spaces~$\modulispace[\semistable]{X}{r}{d}$ and~$\modulispace[\semistable]{X}{r}{L}$ are projective varieties of dimension~$r^2(g-1)+1$ and~$(r^2-1)(g-1)$ respectively.
\end{theorem}

We begin by reducing the projectivity of $\modulispace[\semistable]{X}{r}{d}$ to that of $\modulispace[\semistable]{X}{r}{L}$. The determinant morphism
\[ \det\colon \modulistack[\semistable]{X}{r}{d} \to \stPic_X^d \]
discussed above descends to a map
\[ \det\colon \modulispace[\semistable]{X}{r}{d} \to \Pic_X^d \]
of good moduli spaces whose fiber over $[L] \in \Pic_X^d$ is the moduli space $\modulispace[\semistable]{X}{r}{L}$. It is a classical fact that $\Pic_X^d$ is a projective variety isomorphic to the Jacobian of $X$. Moreover, since $\modulispace[\semistable]{X}{r}{d}$ is a proper algebraic space, to conclude that it is a projective scheme, it suffices by \cite[\spref{0D36}]{stacks-project} to produce a line bundle on it whose restriction to each fiber $\modulispace[\semistable]{X}{r}{L}$ is ample.
We will show that the determinantal line bundle~$\mathrm{L}_V$ will work for an arbitrary vector bundle $V$ satisfying condition \eqref{equation:riemann-roch-condition} below.


From now on, we will fix a line bundle $L$ of degree $d$ on $X$ and focus on the moduli space $\modulispace[\semistable]{X}{r}{L}$.
\paragraph{Setup}
Denote $h = \gcd(r,d)$ and set $r_1 = r/h$ and $d_1 = d/h$. Let~$V$ be a vector bundle on~$X$, with invariants
\begin{equation}
  \label{equation:conditions-on-V}
  \rk(V) = m r_1, \quad \deg(V) = m r_1(g-1) - m d_1
\end{equation}
for some~$m\geq 1$.
By the Riemann--Roch theorem this ensures that~$\chi(X,V\otimes E)=0$ for a vector bundle~$E$ of rank~$r$ and degree~$d$. Hence by Section \ref{section:determinantal} the associated determinantal line bundle~$\mathrm{L}_V$ on~$\modulispace[\semistable]{X}{r}{L}$ depends only on $\rk(V)$ and $\deg(V)$, and comes with a section~$s_V$, depending on~$V$, such that~$s_V([E]) \neq 0$ if and only if~$V\otimes E$ is cohomology-free.

We will prove the theorem using the following two propositions.

\begin{proposition}
  \label{proposition:semiample}
  Let $E$ be a stable vector bundle of rank $r$ and degree $d$. For large enough $m$, a general vector bundle $V$ with
  \[
    \rk(V) = m r_1, \quad \deg(V) = m r_1 (g - 1) - m d_1
  \]
  satisfies
  \begin{equation}
    \HH^0(X, V \otimes E) = \HH^1(X, V \otimes E) = 0.
  \end{equation}
  In particular, the determinantal line bundle~$\mathrm{L}_V$ on $\modulispace[\semistable]{X}{r}{L}$ is \emph{semiample}.
\end{proposition}

\begin{remark}
  In the approach of Faltings this statement (albeit for $\modulispace[\semistable]{X}{r}{d}$) corresponds to \cite[Theorem~I.2]{MR1211997} (note that op.~cit.~is concerned about the more complicated notion of stable Higgs bundles) and \cite[Lemma 3.1]{MR1247504}, but the method of proof will be different.
\end{remark}

\begin{proposition}
  \label{proposition:separates-points}
  Let $E$ and $E'$ be polystable vector bundles of rank $r$ and degree $d$ such that $E$ has a stable summand that is not a summand of $E'$. There exists an integer $m > 0$ and a stable vector bundle $V$ of rank $m r_1$ and degree $m r_1 (g - 1) - m d_1$ such that
  \begin{equation}
    \label{equation:separate-points}
    \HH^0(X, V \otimes E) \neq 0, \quad \HH^0(X, V \otimes E') = 0.
  \end{equation}
  In particular, the sections~$s_V$ of the determinantal line bundle~$\mathrm{L}_V$ on $\modulispace[\semistable]{X}{r}{L}$ \emph{separate points} determined by polystable bundles for which at least one stable summand is different.
\end{proposition}

\begin{remark}
  It is not true 
  that for any two polystable bundles $E$ and $E'$ we can find a vector bundle~$V$ of the appropriate rank and degree for which~$\HH^0(X,V\otimes E)\neq 0$ and~$\HH^0(X,V\otimes E')=0$. Namely, consider two polystable vector bundles
  \[
    E=E_1 \oplus E_1 \oplus E_2 \qquad \text{and}\qquad E'=E_1 \oplus E_2 \oplus E_2,
  \]
  where $E_1$ and $E_2$ are nonisomorphic stable vector bundles of the same rank and degree. However, it can be shown that sections of the form $s_V$ span the space of global sections of $\mathrm{L}_V$ for $m \gg 0$ and induce a closed embedding $\modulispace[\semistable]{X}{r}{d}\hookrightarrow\mathbb{P}^N$; see \cite[Corollary~7.15]{MR2299563}. 
\end{remark}

\begin{remark}
  The method of Faltings (see \cite[Theorem~I.4]{MR1211997} and \cite[Lemma~4.2]{MR1247504}) takes a different approach at this point. He shows that, given a map $C \to \modulistack[\semistable]{X}{r}{d}$ from a smooth projective curve $C$ corresponding to a family~$\mathcal{E}$ of semistable vector bundles on~$X$, the degree of the determinantal line bundle~$\mathcal{L}_V$ 
  is nonnegative and is zero if and only if all the fibers~$\{\mathcal{E}_t\}_{t\in C}$ are S-equivalent. This allows one to conclude that the morphism $\modulispace[\semistable]{X}{r}{d} \to \mathbb{P}^N$ has finite fibers, and then the argument can proceed as in the proof of Theorem \ref{theorem:projectivity}. Interestingly, Faltings's argument uses stability of vector bundles on the curve $C$. 

\end{remark}

%
%

\begin{proof}[Proof of Theorem \ref{theorem:projectivity}]
  By Proposition \ref{proposition:semiample} the determinantal line bundle~$\mathcal{L}_V$ associated to the appropriate choice of~$V$ is semiample, so choose some power for which it is basepoint-free and obtain a morphism~$\phi\colon\modulispace[\semistable]{X}{r}{d}\to\mathbb{P}^N$ for some~$N$.

  By Proposition \ref{proposition:separates-points} the morphism is actually quasi-finite. Indeed, for a given polystable bundle $E = E_1^{\oplus r_1} \oplus \cdots \oplus E_s^{\oplus r_s}$ where the $E_1, \ldots, E_s$ are nonisomorphic stable bundles, there are only finitely many other polystable bundles whose stable summands are precisely $E_1, \ldots, E_s$, and for any other polystable bundle $E'$, there exists a section $s_V$ separating $E$ and $E'$.

  Since $\modulispace[\semistable]{X}{r}{d}$ is a proper algebraic space, the morphism $\modulispace[\semistable]{X}{r}{d} \to \mathbb{P}^N$ is proper by \cite[\spref{04NX}]{stacks-project} and hence finite by \cite[\spref{0A4X}]{stacks-project}. But then it is also affine, hence representable, so~$\modulispace[\semistable]{X}{r}{d}$ is a scheme. Finally, pulling back an ample line bundle on~$\mathbb{P}^N$ along the affine morphism produces an ample line bundle on $\modulispace[\semistable]{X}{r}{d}$, by \cite[\spref{0892}]{stacks-project}.
\end{proof}

We are left with proving Propositions \ref{proposition:semiample} and \ref{proposition:separates-points}. Their proofs will be based on a dimension-counting argument inspired by \cite{MR2068965} (see also \cite[Section 5]{MR0696517} for similar methods), where the authors actually prove a more refined result that yields an effective bound on the rank of $V$ in the theorem.
To get started, we will need the following lemma, discussed in \cite[Remark 4.2]{MR1487229}.

\begin{lemma}
  \label{lemma:bounded-family}
  If $\{F_\al\}_{\al \in \Lambda}$ is a bounded family of vector bundles of rank $r$ on $X$, there exists a scheme $S$ of finite type over $k$ of dimension at most $r^2(g-1) + 1$ and a vector bundle $\sF$ on $S \times X$ such that each $F_\al$ appears as the fiber of $\sF$ over some point $s \in S$.
\end{lemma}

In fact, the stable bundles of rank $r$ and degree $d$ can be parameterized by a smooth variety of dimension $r^2(g-1) + 1$, and any bounded family of non-stable bundles by a variety of dimension at most $r^2(g-1)$.

The following lemma will be the technical heart for the proof of both propositions. We will apply it for~$\epsilon=-1,0,1$, and think of it as a ``fudge factor'' perturbing the equality~$\chi(X,E\otimes V)=0$ for~$\epsilon=0$.

\begin{lemma}
  \label{lem:generic-image-slope}
  Let $E$ be a vector bundle of rank $r$ and degree $d$ on $X$, and let $\eps$ be an integer. Let $m$ be a sufficiently large integer and $V$ a general stable vector bundle with
      \[ \rk(V) = m r_1, \quad \deg(V) = m r_1 (g-1) - m d_1 + \eps. \]
  \begin{enumerate}
    \item\label{enumerate:generic-image-slope-1} There exist no maps $V^\vee \to E$ whose image $F$ has slope $\mu(F) < \mu(E)$.

    \item\label{enumerate:generic-image-slope-2} If moreover $E$ is stable and $\eps \le 0$, there exist no nonzero maps $V^\vee \to E$. 
  \end{enumerate}
\end{lemma}

\begin{proof}
  We can assume that $\rk(V) > r$. Note that $V^\vee$ satisfies
  \[ \rk(V^\vee) = m r_1, \quad \deg(V^\vee) = m r_1 (1-g) + m d_1 - \eps. \]
  If $V^\vee \to E$ is a map with image $F$, then, since $V^\vee$ is stable, we must have
  \[ \mu(F) > \mu(V^\vee) = \mu(E) - g + 1 - \frac{\eps}{m r_1} \ge \mu(E) - g + 1 - \eps. \]
  Thus, there are only finitely many options for the slope of the image $F$ such that $\mu(F) \le \mu(E)$. Since moreover we must have $1 \le \rk(F) \le r$, there are only finitely many options for the rank and degree of $F$.

  To prove \ref{enumerate:generic-image-slope-1}, fix integers $1 \le k \le r$ and $l$ with $l/k < \mu(E)$. We will find an upper bound for the dimension of a variety that parameterizes all stable bundles $V$ such that there exists a map $V^\vee \to E$ whose image $F$ has rank $k$ and degree $l$. Any such $V^\vee$ fits into an exact sequence
  \begin{equation}\label{Vdualses}
    \begin{tikzcd}
      0 \arrow[r] & G \arrow[r] & V^\vee \arrow[r] \arrow[rd] & F \arrow[r] \arrow[d, hook] & 0 \\
      & & & E
    \end{tikzcd}
  \end{equation}
  where $G$ satisfies
  \[ \rk(G) = m r_1 - k, \quad \deg(G) = m r_1 (1-g) + m d_1 - \eps - l. \]
  The possible sheaves $F$ that can appear in \eqref{Vdualses} form a bounded family since they are all subsheaves of $E$ of fixed rank and degree. By Lemma \ref{lemma:bounded-family} they are parameterized by a scheme $\sU_1$ of finite type over $k$ with
  \[ \dim \sU_1 \le k^2(g-1) + 1. \]
  Similarly, since the family of stable bundles of fixed rank and degree is bounded, the sheaves $G$ that can appear in \eqref{Vdualses} form a bounded family and so by Lemma \ref{lemma:bounded-family} are parameterized by a scheme $\sU_2$ with
  \[ \dim \sU_2 \le (m r_1 - k)^2 (g-1) + 1. \]
  Note that since $V^\vee$ in \eqref{Vdualses} is stable, we must have $\Hom_X(F, G) = 0$, because for any nonzero map $F \to G$ the composition
  \[ V^\vee \twoheadrightarrow F \to G \hookrightarrow V^\vee \]
  would give a nonzero endomorphism of $V^\vee$ that is not an isomorphism. Now, for a fixed $F$ and $G$, the possible $V^\vee$ appearing in \eqref{Vdualses} are parameterized by an open subset of $\mathbb{P}\mathrm{Ext}_X^1(F, G)$. By the Riemann--Roch theorem we have
  \begin{align*}
      &\dim \Ext_X^1(F, G) \\
      &\quad = - \chi(X, F^\vee \otimes G) = -\deg(F^\vee \otimes G) - \rk(F^\vee \otimes G) (1-g) \\
      &\quad = \deg(F)\rk(G) - \rk(F) \deg(G) + \rk(F) \rk(G) (g - 1) \\
      &\quad = l(m r_1 - k) - k(m r_1 (1-g) + m d_1 - \eps - l) + k(m r_1 - k)(g - 1) \\
      &\quad = (l r_1 - k d_1) m + 2 k r_1 (g-1) m - k^2 (g-1) + k\eps.
  \end{align*}
  Note that by our assumption on $k$ and $l$, we have
  \[ l r_1 - k d_1 = k r_1 \left(\frac{l}{k} - \frac{d_1}{r_1}\right) = k r_1 (\mu(F) - \mu(E)) < 0. \]
  Thus, the possible $V^\vee$ that could appear are parameterized by an open subset of a projective bundle $\sP$ over an open subset $\sU \subs \sU_1 \times \sU_2$ of dimension
  \begin{align*}
      \dim \sP & \le \dim \sU_1 + \dim \sU_2 + \dim \Ext_X^1(F, G) - 1 \\
      & \le k^2 (g - 1) + 1 + (m r_1 - k)^2 (g - 1) + 1 \\
      & \quad + k r_1 (\mu(F) - \mu(E)) m + 2 k r_1 (1-g) m - k^2 (g-1) - k\eps - 1 \\
      & = (m r_1)^2 (g - 1) + k r_1 (\mu(F) - \mu(E)) m + k^2 (g - 1) + k\eps + 1.
  \end{align*}
  Since the coefficient of $m$ in the last expression is negative, we see that for large enough $m$, we have $\dim \sP < (m r_1)^2 (g - 1) + 1$, where by Theorem \ref{theorem:good-moduli-space} the right-hand side is the dimension of the moduli space of stable bundles of rank $m r_1$ and fixed degree. Thus, a general stable bundle $V$ cannot appear in an extension of the form \eqref{Vdualses}. This proves \ref{enumerate:generic-image-slope-1}.

  To prove \ref{enumerate:generic-image-slope-2}, assume that $E$ is stable and $\eps \le 0$. By \ref{enumerate:generic-image-slope-1}, if $V$ is a general stable vector bundle of the given rank and degree and $m$ is sufficiently large, there are no maps $V^\vee \to E$ whose image has slope less than $\mu(E)$. Thus, since $E$ is stable, any nonzero map must be surjective, and so we are led to estimate the dimension of a variety parameterizing all stable bundles $V$ such that there exists a short exact sequence
  \[ 0 \to G \to V^\vee \to E \to 0. \]
  Such $V^\vee$ are parameterized by a projective bundle $\sP$ over a scheme $\sU_2$, where $\sU_2$ parameterizes bundles $G$ with
  \[ \rk(G) = m r_1 - r, \quad \deg(G) = m r_1 (1-g) + m d_1 - \eps - d \]
  and again $\Hom_X(E, G) = 0$ since $V^\vee$ is stable. The fiber of $\sP$ over $[G] \in \sU_2$ is $\mathbb{P}\mathrm{Ext}_X^1(E, G)$, and the dimension calculations from above apply, so we have
  \begin{align*}
      \dim \sP & \le \dim \sU_2 + \dim \Ext_X^1(E, G) - 1 \\
      & = (m r_1 - r)^2 (g - 1) + 1 + 2 r r_1 (g-1) m - r^2 (g-1) + r\eps - 1 \\
      & = (m r_1)^2 (g - 1) + r\eps.
  \end{align*}
  Since $\eps \le 0$, we have $\dim \sP < (m r_1)^2 (g - 1) + 1$, where the left-hand side is the dimension of the moduli space of stable bundles of rank $m r_1$ by Theorem \ref{theorem:good-moduli-space}.
\end{proof}

\begin{proof}[Proof of Proposition \ref{proposition:semiample}]
  Applying Lemma \ref{lem:generic-image-slope} with $\epsilon = 0$ to a stable vector bundle $E$ produces an open subset in the relevant moduli space such that any $V$ in the open subset satisfies the first claim of the proposition.

  Now let~$E$ be a polystable vector bundle corresponding to a point in~$\modulispace[\semistable]{X}{r}{L}$. 
  Taking the intersection of the open subsets for all the stable summands of $E$ produces a nonempty open subset where one can find a vector bundle~$V$ such that~$\HH^i(X,E\otimes V)=0$ for~$i=0,1$. Hence the characterization from Proposition \ref{proposition:determinantal-line-bundle-properties} shows that the associated section~$s_V$ does not vanish at~$[E]\in\modulispace[\semistable]{X}{r}{L}$. 

\end{proof}

\begin{lemma}
  \label{lemma:general-stable-surjective}
  Let $E_0, E_1, \ldots, E_n$ be stable vector bundles of slope $\mu = r/d$. Denote the rank of $E_i$ by $k_i$. For $m$ sufficiently large, a general stable vector bundle $V$ with
  \[
    \rk(V) = m r_1 \quad \text{and} \quad \deg(V) = m r_1 (g - 1) - m d_1 + 1
  \]
  satisfies the following:
  \begin{enumerate}[label={\upshape(\roman*)}]
    \item\label{enumerate:general-stable-surjective-1} $\dim \Hom_X(V^\vee, E_i) = k_i$, or equivalently $\Ext_X^1(V^\vee,E_i)=0$;
    \item\label{enumerate:general-stable-surjective-2} for $0 \le i \le n$, every nonzero map $V^\vee \to E_i$ is surjective;
    \item\label{enumerate:general-stable-surjective-3} for $1 \le i \le n$, for all nonzero maps $V^\vee \to E_0$ and $V^\vee \to E_i$, the induced map $V^\vee \to E_0 \oplus E_i$ is surjective.
  \end{enumerate}
\end{lemma}

\begin{proof}
  Note first that $\mu = d_1/r_1 = \deg(E_i)/k_i$. We have
  \[
    \dim \Hom_X(V^\vee, E_i) = \hh^0(V \otimes E_i) = \chi(X, V \otimes E_i) + \hh^1(X, V \otimes E_i),
  \]
  and by the Riemann--Roch theorem,
  \begin{align*}
    \chi(X, V \otimes E_i) & = \deg(V \otimes E_i) + \rk(V \otimes E_i) (1 - g) \\
    & = \deg(V)\rk(E_i) + \rk(V) \deg(E_i) + \rk(V) \rk(E_i) (1 - g) \\
    & = (m r_1 (g - 1) - m d_1 + 1) k_i + m r_1 k_i \mu + m r_1 k_i (1 - g) \\
    & = k_i.
  \end{align*}
  To obtain \ref{enumerate:general-stable-surjective-1} we must show that for a general stable $V$ we have $\HH^1(X, V \otimes E_i) = 0$ for $0\leq i\leq n$. By Serre duality, we have
  \[ \hh^1(X, V \otimes E_i) = \hh^0(X, V^\vee \otimes \om_X \otimes  E_i^\vee) = \dim \Hom_X(V \otimes \om_X^\vee, E_i^\vee). \]
  Now $V^\vee \otimes \om_X$ is a general stable bundle of rank $m r_1$ and degree
  \begin{align*}
    \deg(V^\vee \otimes \omega_X) & = -\deg(V) + \rk(V) \deg(\omega_X) \\
    & = m r_1 (1 - g) + m d_1 - 1 + m r_1(2 g - 2) \\
    & = m r_1(g - 1) + m d_1 - 1,
  \end{align*}
  and $E_i^\vee$ is a stable bundle of slope $-d_1/r_1$, so we may apply Lemma \ref{lem:generic-image-slope}(ii) with $\epsilon = -1$ to each $E_i^\vee$ to conclude that, for sufficiently large $m$ and general stable $V$, we have $\Hom_X(V \otimes \om_X^\vee, E_i^\vee) = 0$ for each $i$. This gives \ref{enumerate:general-stable-surjective-1}.

  By applying Lemma \ref{lem:generic-image-slope}\ref{enumerate:generic-image-slope-1} with $\epsilon = 1$ to each $E_i$ and by increasing $m$ if necessary, we may assume that for general stable $V$ the image of every nonzero map $V^\vee \to E_i$ has slope equal to $\mu$, and, since $E_i$ is stable, such a map must be surjective. This gives \ref{enumerate:general-stable-surjective-2}.

  Finally, by applying Lemma \ref{lem:generic-image-slope}(i) to each $E_0 \oplus E_i$ for $1 \le i \le n$ and increasing $m$ if necessary, we may assume that for general stable $V$, every nonzero map $V^\vee \to E_0 \oplus E_i$ has image with slope equal to $\mu$. For such a $V$ and two nonzero maps $V^\vee \to E_0$ and $V^\vee \to E_i$, the image $F \subs E_0 \oplus E_i$ of the induced map $V^\vee \to E_0 \oplus E_i$ must thus have slope $\mu(F) = \mu$, and moreover $F$ surjects onto both summands. But the only nonzero subsheaves of $E_0 \oplus E_i$ of slope $\mu$ are $E_0$, $E_i$, and $E_0 \oplus E_i$ and, since $E_0$ and $E_i$ are by assumption nonisomorphic, we must have $F = E_0 \oplus E_i$. This gives \ref{enumerate:general-stable-surjective-3}.
\end{proof}

Before we give the proof of Proposition \ref{proposition:separates-points} we introduce a method for constructing new vector bundles out of old ones.
\begin{definition}
  Let~$E$ be a vector bundle on~$X$ and let~$x\in X(k)$ be a closed point. Denote by~$\mathcal{O}_x$ the skyscraper sheaf at~$x$ and $E(x)$ the fiber of $E$ at~$x$. Let~$\alpha\in\Hom_X(E,\mathcal{O}_x)=E(x)^\vee$ be a nonzero morphism. The \emph{elementary transformation} of~$E$ at~$x$ associated to~$\alpha$ is the vector bundle~$E'$ defined as the kernel of~$\alpha$, so that we have the short exact sequence
  \begin{equation}
    0\to E'\to E\overset{\alpha}{\to}\mathcal{O}_x\to 0.
  \end{equation}
\end{definition}
The kernel of~$\alpha$ does not change under rescaling, so elementary transformations of~$E$ at~$x$ correspond to closed points of~$\mathbb{P}E(x)^\vee$.


If~$\phi\colon F\to E$ is a morphism, then~$\operatorname{im}\phi(x)$ is a subspace of~$E(x)$, while~$\ker \alpha(x)$ is a hyperplane in~$E(x)$, which we will denote by~$H$. If~$\operatorname{im}\phi(x)\subseteq H$ then the composition~$\smash{F\xrightarrow{\phi} E\xrightarrow{\alpha}\mathcal{O}_x}$ vanishes and so~$F\to E$ factors through~$E'$. Conversely, if~$\operatorname{im}\phi(x)$ is not contained in~$H$ then~$F\to E$ does not factor through~$E'$.


\begin{proof}[Proof of Proposition \ref{proposition:separates-points}]
  Let $E_1, \ldots, E_n$ be the stable summands of $E'$, and let $E_0$ be a stable summand of $E$ not isomorphic to any of the $E_i$. As before, let $k_i$ denote the rank of $E_i$. It suffices to find a vector bundle $F$ such that
  \[ \HH^0(X, F \otimes E_0) \neq 0 \quad \text{but} \quad \HH^0(X, F \otimes E_i) = 0 \; \text{for} \; i = 1,\ldots, n. \]
  Let $V$ be a stable vector bundle satisfying the conditions of Lemma \ref{lemma:general-stable-surjective}, and fix a nonzero map $\phi\colon E_0^\vee \to V$. The goal is to find a subsheaf $F \subset V$ of degree $\deg(F) = \deg(V) - 1$ such that

  \begin{itemize}
    \item $\phi$ factors through $F$, but
    \item for $i = 1, \ldots, n$, no nonzero map $E_i^\vee \to V$ factors through $F$.
  \end{itemize}

  Let $Q$ denote the cokernel of $\phi$. It is locally free since by Lemma \ref{lemma:general-stable-surjective}(ii) $\phi$ is the dual of a surjective map $V^\vee \to E_0$. Note that, for any nonzero map $\psi_i\colon E_i^\vee \to V$, the induced map $E_0^\vee \oplus E_i^\vee \to V$ is dual to a surjection by Lemma \ref{lemma:general-stable-surjective}(iii) and hence injective. The inclusions
  \[ E_0^\vee \hookrightarrow E_0^\vee \oplus E_i^\vee \overset{(\phi,\psi_i)}{\hookrightarrow} V \]
  induce a short exact sequence
  \[ 0 \to E_i^\vee \xrightarrow{\overline{\psi}_i} Q \to V/E_0^\vee \oplus E_i^\vee \to 0, \]
  and so the composition $\overline{\psi}_i\colon E_i^\vee \to V \to Q$ is injective with locally free cokernel, from which it follows that the restriction of $\overline{\psi}_i$ to a fiber is also injective.

  Let $x \in X$ be a closed point, let $Q(x)$ and $V(x)$ denote the fibers of $Q$ and $V$ at $x$ respectively, and let
  $\mathrm{Grass}(k_i, Q(x))$ denote the Grassmannian of $k_i$-dimensional subspaces of the vector space $Q(x)$. Let $M_i \subs \mathrm{Grass}(k_i, Q(x))$ denote the image of the morphism
  \[ \mathbb{P}\mathrm{Hom}_X(E_i^\vee, V) \to \mathrm{Grass}(k_i, Q(x)) \]
  that sends a nonzero map $\psi_i$ to its image under the composition
  \[ E_i^\vee \xrightarrow{\overline{\psi}_i} Q \to Q(x). \]
  This morphism is well defined, because in the previous paragraph we established that~$\overline{\psi}_i(x)$ is injective.
  We have $\dim M_i \le k_i - 1$ by Lemma~\ref{lemma:general-stable-surjective}\ref{enumerate:general-stable-surjective-1}.

  For a hyperplane $H \subs Q(x)$, let $Z_{H,i} \subs \mathrm{Grass}(k_i, Q(x))$ denote the Schubert variety parameterizing subspaces contained in $H$. The codimension of $Z_{H,i}$ in $\mathrm{Grass}(k_i, Q(x))$ is $k_i$, and so the expected dimension of the intersection $M_i \cap Z_{H,i}$ is $-1$. Thus, by the Bertini--Kleiman theorem \cite[Corollary 4(i)]{MR360616} we obtain for the general hyperplane $H$ that $M_i \cap Z_{H,i}$ is empty for $i = 1, \ldots, n$. Let us fix one such hyperplane $H$.

  Let $F$ denote the elementary transformation of~$V$ at~$x$ defined by the subspace~$H$, i.e.,~it is the kernel of the morphism $V\to\mathcal{O}_x$ determined by the composition
  \[ V \to V(x) \to Q(x) \to Q(x)/H. \]
  By construction $\phi\colon E_0^\vee \to V$ factors through $F$, and so
  \[ \HH^0(X, E_0 \otimes F) = \Hom_X(E_0^\vee, F) \neq 0. \]
  On the other hand, no nonzero map $E_i^\vee \to V$ factors through $F$, as the composition
  \[ E_i^\vee(x) \xrightarrow{\psi_i(x)}  V(x) \to Q(x)/H \]
  is nonzero by the choice of $H$. Thus,
  \[ \HH^0(X, E_i \otimes F) = \Hom_X(E_i^\vee, F) = 0 \quad \mathrm{for} \quad i = 1, \ldots, n. \]
  Since $\rk(F) = \rk(V)$ and $\deg(F) = \deg(V) - 1$, this completes the proof.
\end{proof}

\paragraph{Acknowledgements}
We thank Hannah Larson and Shizhang Li for being part of our group during SPONGE, and the many interesting conversations during and after the workshop. We are grateful for the referee's comments.
PB and TT would also like to thank Chiara Damiolini, Hans Franzen, Vicky Hoskins, and Svetlana Makarova for interesting discussions on a GIT-free proof of the (quasi-)projectivity of moduli of semistable quiver representations.
JA was partially supported by NSF grants DMS-1801976 and DMS-2100088.
DB was supported by NSF Postdoctoral Research Fellowship DMS-1902875.


\bibliographystyle{plainurl}
\urlstyle{tt}
\renewcommand\path[1]{\texttt{#1}}
{\small\bibliography{alper}}

{\small
\setlength\parindent{0cm}
\setlength\parskip{3pt}

Department of Mathematics, University of Washington, Box 354350, C-138 Padelford, Seattle, WA~98195-4350, United States

Department of Mathematics, University of Luxembourg, 6, Avenue de la Fonte, L-4364 Esch-sur-Alzette, Luxembourg

Department of Mathematics, University of California, Berkeley, 970 Evans Hall, Berkeley, CA~94720-3840, United States

Department of Mathematics, University of Michigan, 530~Church Street, Ann Arbor, MI~48109, United States

Matematiska institutionen, Stockholm Universitet, Kr\"aftriket 5A, 114 19 Stockholm, Sweden

\texttt{alper@uw.edu},
\texttt{pieter.belmans@uni.lu},
\texttt{braggdan@berkeley.edu},
\texttt{liangji@umich.edu},
\texttt{tuomas.tajakka@math.su.se}
}

\end{document}